\documentclass{amsart}
\usepackage{amsmath}
\usepackage{amssymb}
\usepackage{amsfonts}

\newtheorem{theorem}{Theorem}[section]
\newtheorem{corollary}[theorem]{Corollary}
\newtheorem{lemma}[theorem]{Lemma}
\newtheorem{definition}[theorem]{Definition}
\newtheorem{proposition}[theorem]{Proposition}
\newtheorem{remark}[theorem]{Remark}
\numberwithin{theorem}{section}
\newtheorem{acknowledgement}{Acknowledgement}

\begin{document}
\title[Extremal marginals of a local $\mathcal{CCP}$-map ]{Extremal
marginals of an unbounded local completely positive and local completely
contractive map }
\author{Maria Joi\c{t}a}
\address{Department of Mathematics, Faculty of Applied Sciences, University
Politehnica of Bucharest, 313 Spl. Independentei, 060042, Bucharest, Romania
and \\
Simion Stoilow Institute of Mathematics of the Romanian Academy, P.O. Box
1-764, 014700, Bucharest, Romania}
\email{mjoita@fmi.unibuc.ro and maria.joita@mathem.pub.ro}
\urladdr{http://sites.google.com/a/g.unibuc.ro/maria-joita/}
\thanks{This work was partially supported by a grant of the Ministry of
Research, Innovation and Digitization, CNCS/CCCDI--UEFISCDI, project number
PN-III-P4-ID-PCE-2020-0458, within PNCDI III}

\begin{abstract}
In this paper we consider the unbounded local completely positive and local
completely contractive maps on maximal tensor product of unital locally $%
C^{\ast }$-algebras and discuss on extremal points of certain convex subsets
in the set of such maps.
\end{abstract}

\maketitle

\section{Introduction}

Locally \ $C^{\ast }$-algebras are generalizations of $C^{\ast }$-algebras,
the topology on a locally $C^{\ast }$-algebra is defined by an upward
directed family of $C^{\ast }$-seminorms. The concrete models for locally $%
C^{\ast }$-algebras are $\ast $-algebras of unbounded linear operators on a
Hilbert space. In the literature, the locally $C^{\ast }$-algebras are
studied under different names like pro-$C^{\ast }$-algebras (D. Voiculescu
\cite{V}, N.C. Philips\ \cite{Ph}), $LMC^{\ast }$-algebras ( K. Schm\"{u}%
dgen \cite{S}), $b^{\ast }$-algebras (C. Apostol \cite{A}) and multinormed $%
C^{\ast }$-algebras (A. Dosiev \cite{D1}). The term locally \ $C^{\ast }$%
-algebra is due to A. Inoue \cite{I}.

A locally $C^{\ast }$-algebra is a complete Hausdorff complex topological $%
\ast $-algebra $\mathcal{A}$ whose topology is determined by an upward
filtered family $\{p_{\lambda }\}_{\lambda \in \Lambda }\ $of $C^{\ast }$%
-seminorms defined on $\mathcal{A}$. An element $a\in \mathcal{A}$ is
positive if there exists $b\in \mathcal{A}$ such that $a=b^{\ast }b$ and it
is local positive if there exist $c,d\in \mathcal{A}$ and $\lambda \in
\Lambda $ such that $a=c^{\ast }c+d$ and $p_{\lambda }\left( d\right) =0$.
Thus, the notion of (local) completely positive maps appeared naturally
while studying linear maps between locally $C^{\ast }$-algebras. A quantized domain in a Hilbert space $\mathcal{H}$ is a triple $\{%
\mathcal{H};\mathcal{E};\mathcal{D}_{\mathcal{E}}\}$, where $\mathcal{E}=\{%
\mathcal{H}_{\iota };\iota \in \Upsilon \}$ is an upward filtered family of
closed subspaces such that the union space $\mathcal{D}_{\mathcal{E}%
}=\bigcup\limits_{\iota \in \Upsilon }\mathcal{H}_{\iota }$ is dense in $%
\mathcal{H\ }$\cite{D1}. The collection of all linear operators $T:\mathcal{D}%
_{\mathcal{E}}\rightarrow \mathcal{D}_{\mathcal{E}}$ such that $T(\mathcal{H}%
_{\iota })\subseteq \mathcal{H}_{\iota },T(\mathcal{H}_{\iota }^{\bot }\cap 
\mathcal{D}_{\mathcal{E}})\subseteq \mathcal{H}_{\iota }^{\bot }\cap 
\mathcal{D}_{\mathcal{E}}$ and $\left. T\right\vert _{\mathcal{H}_{\iota
}}\in B(\mathcal{H}_{\iota })$ for all $\iota \in \Upsilon ,\ $denoted by $%
C^{\ast }(\mathcal{D}_{\mathcal{E}}),$ has a structure of locally $C^{\ast }$%
-algebra\ with the topology determined by the family of $C^{\ast }$-seminorms  $%
\{\left\Vert \cdot \right\Vert _{\iota }\}_{\iota \in \Upsilon },$ where $%
\left\Vert T\right\Vert _{\iota }=\left\Vert \left. T\right\vert _{\mathcal{H%
}_{\iota }}\right\Vert .$
In \cite%
{MJ3}, we showed that an unbounded map is local completely positive if and
only if it is continuous and completely positive. The structure of the
unbounded local completely positive and local completely contractive maps on
unital locally $C^{\ast }$-algebras is described in \cite{D1}. More exactly,
Dosiev proved a Stinespring type theorem for unbounded local $\mathcal{CPCC}$
(completely positive and completely contractive) maps on unital locally $%
C^{\ast }$-algebras. Bhat, Ghatak and Pamula \cite{BGK} introduced a partial
order relation on the set of all unbounded local $\mathcal{CPCC}$-maps on a
unital locally $C^{\ast }$-algebra $\mathcal{A}$ and described this relation
in terms of their minimal Stinespring dilations. In \cite{MJ1}, we
determined the structure of unbounded local $\mathcal{CPCC}$-maps on unital
locally $C^{\ast }$-algebras which preserve the local orthogonality and in
\cite{MJ3}, we proved some factorization properties for unbounded local
positive maps. In \cite{MJ2}, we proved a local convex version of
Wittstock's extension theorem and a Stinespring type theorem for unbounded
local $\mathcal{CC}$ (completely contractive) maps. In this paper, we
consider the unbounded local $\mathcal{CPCC}$-maps on maximal tensor product
of unital locally $C^{\ast }$-algebras and discuss on extremal points of
certain convex subsets in the set of such maps. In \cite{HHP}, Haapasalo, Heinosaari and Pellonp\"{a}\"{a} considered
bounded $\mathcal{CP\ }$(completely positive) maps on tensor product of von
Neumann algebras and showed that if one of the marginal maps of such map is
an extremal point, then the marginal maps uniquely determine the map. We
extend the results of \cite{HHP} to unbounded local $\mathcal{CPCC}$-maps on
maximal tensor product of unital locally $C^{\ast }$-algebras.

The paper is organized as follows. In Section $2$, we gather some basic
facts on locally $C^{\ast }$-algebras, concrete models for locally $C^{\ast }
$-algebras and unbounded local $\mathcal{CPCC}$-maps that are needed for
understanding the main results of the paper. Section $3$ is devoted to the
characterization of the extremal elements in certain convex subsets of $%
\mathcal{CPCC}_{\text{loc}}(\mathcal{A},C^{\ast }(\mathcal{D}_{\mathcal{E}}))
$. In Section $4$, we show that two unital unbounded local $\mathcal{CPCC}$%
-maps on the unital locally $C^{\ast }$-algebras $\mathcal{A}$ and $\mathcal{B}$%
, whose ranges commute, induce an unbounded local $\mathcal{CPCC}$-map an
the maximal tensor product $\mathcal{A\otimes }_{\max }\mathcal{B}$ of $%
\mathcal{A}$ and $\mathcal{B}$ (Proposition \ref{tmax}). Section $5$ is
devoted to marginal maps and joint maps. The structure of an unbounded local
$\mathcal{CPCC}$-map on the maximal tensor product of locally $C^{\ast }$%
-algebras in terms of a minimal Stinespring dilation of one of its marginal
map is given in Theorem \ref{Help}. As in the case of bounded $\mathcal{CP}$%
-maps, we show that if one of the marginal maps is an extremal point, then
the joint map is unique (Theorem \ref{C} $(1)$) and when both marginal maps
are extremal points, then the unique joint map is extremal (Theorem \ref{C} $%
(2)$).

\section{Preliminaries}

\paragraph{\textbf{Locally }$C^{\ast }$\textbf{-algebras}}

A \textit{locally }$C^{\ast }$\textit{-algebra }is a complete Hausdorff
complex topological $\ast $-algebra $\mathcal{A}$ whose topology is
determined by an upward filtered family $\{p_{\lambda }\}_{\lambda \in
\Lambda }\ $of $C^{\ast }$-seminorms defined on $\mathcal{A}$.

An element $a\in \mathcal{A}$ is \textit{bounded} if $\sup \{p_{\lambda
}\left( a\right) ;\lambda \in \Lambda \}<\infty .$ Then $b\left( \mathcal{A}%
\right) =\{a\in \mathcal{A};\left\Vert a\right\Vert _{\infty }:=\sup
\{p_{\lambda }\left( a\right) ;\lambda \in \Lambda \}<\infty \}\ $is a $%
C^{\ast }$-algebra with respect to the $C^{\ast }$-norm $\left\Vert \cdot
\right\Vert _{\infty }$. Moreover, $b\left( \mathcal{A}\right) $ is dense in
$\mathcal{A}$ \cite[Proposition 1.11]{Ph}$.$

We see that $\mathcal{A}$ can be realized as a projective limit of an
inverse family of $C^{\ast }$-algebras as follows: For each $\lambda \in
\Lambda $, let $\mathcal{I}_{\lambda }=\{a\in \mathcal{A};p_{\lambda }\left(
a\right) =0\}$. Clearly, $\mathcal{I}_{\lambda }$ is a closed two sided $%
\ast $-ideal in $\mathcal{A}$ and $\mathcal{A}_{\lambda }=\mathcal{A}/%
\mathcal{I}_{\lambda }$ is a $C^{\ast }$-algebra with respect to the norm
induced by $p_{\lambda }$ \cite[Corollary 1.12]{Ph}. The canonical quotient $%
\ast $-morphism from $\mathcal{A\ }$to $\mathcal{A}_{\lambda }$ is denoted
by $\pi _{\lambda }^{\mathcal{A}}$. For each $\lambda _{1},\lambda _{2}\in
\Lambda $ with $\lambda _{1}\leq \lambda _{2}$, there is a canonical
surjective $\ast $-morphism $\pi _{\lambda _{2}\lambda _{1}}^{\mathcal{A}}:$
$\mathcal{A}_{\lambda _{2}}\rightarrow \mathcal{A}_{\lambda _{1}}$ given by $%
\pi _{\lambda _{2}\lambda _{1}}^{\mathcal{A}}\left( a+\mathcal{I}_{\lambda
_{2}}\right) =a+\mathcal{I}_{\lambda _{1}},$ $a\in \mathcal{A}$. Then, $\{%
\mathcal{A}_{\lambda },\pi _{\lambda _{2}\lambda _{1}}^{\mathcal{A}}\}$\
forms an inverse system of $C^{\ast }$-algebras, since $\pi _{\lambda _{1}}^{%
\mathcal{A}}=$ $\pi _{\lambda _{2}\lambda _{1}}^{\mathcal{A}}\circ \pi
_{\lambda _{2}}^{\mathcal{A}}$ whenever $\lambda _{1}\leq \lambda _{2}$. The
projective limit%
\begin{equation*}
\lim\limits_{\underset{\lambda }{\leftarrow }}\mathcal{A}_{\lambda
}=\{\left( a_{\lambda }\right) _{\lambda \in \Lambda }\in
\prod\limits_{\lambda \in \Lambda }\mathcal{A}_{\lambda };\pi _{\lambda
_{2}\lambda _{1}}^{\mathcal{A}}\left( a_{\lambda _{2}}\right) =a_{\lambda
_{1}}\text{ whenever }\lambda _{1}\leq \lambda _{2},\lambda _{1},\lambda
_{2}\in \Lambda \}
\end{equation*}%
of the inverse system of $C^{\ast }$-algebras $\{\mathcal{A}_{\lambda },\pi
_{\lambda _{2}\lambda _{1}}^{\mathcal{A}}\}$ is a locally $C^{\ast }$%
-algebra that can be identified with $\mathcal{A}$ by the map $a\mapsto
\left( \pi _{\lambda }^{\mathcal{A}}\left( a\right) \right) _{\lambda \in
\Lambda }.$

\paragraph{\textbf{Positive and local positive elements}}

Let $\mathcal{A}$ be a locally $C^{\ast }$-algebra with the topology defined
by the family of $C^{\ast }$-seminorms $\left\{ p_{\lambda }\right\}
_{\lambda \in \Lambda }.$

An element $a\in \mathcal{A}$ is \textit{self-adjoint} if $a^{\ast }=a$ and
it is\textit{\ positive} if $a=b^{\ast }b$ for some $b\in \mathcal{A}.$

An element $a\in \mathcal{A}$ is called \textit{local self-adjoint} if $%
a=a^{\ast }+c$, where $c\in \mathcal{A}$ such that $p_{\lambda }\left(
c\right) =0$ for some $\lambda \in \Lambda $, and we call $a$ as $\lambda $%
-self-adjoint, and \textit{local positive} if $a=b^{\ast }b+c$, where $%
b,c\in $ $\mathcal{A}$ such that $p_{\lambda }\left( c\right) =0\ $ for some
$\lambda \in \Lambda $, we call $a$ as $\lambda $-positive and note $a\geq
_{\lambda }0$. We write $a=_{\lambda }0$ whenever $p_{\lambda }\left(
a\right) =0$. Note that $a\in \mathcal{A}$ is local self-adjoint if and only
if there is $\lambda \in \Lambda $ such that $\pi _{\lambda }^{\mathcal{A}%
}\left( a\right) $ is self adjoint in $\mathcal{A}_{\lambda }$ and $a\in
\mathcal{A}$ is local positive if and only if there is $\lambda \in \Lambda $
such that $\pi _{\lambda }^{\mathcal{A}}\left( a\right) $ is positive in $%
\mathcal{A}_{\lambda }.$

\paragraph{\textbf{Quantized domains}}

Throughout the paper, $\mathcal{H}$ is a complex Hilbert space and $B(%
\mathcal{H})$ is the algebra of all bounded linear operators on the Hilbert
space $\mathcal{H}$.

Let $(\Upsilon ,\leq )$ be a directed poset. A \textit{quantized domain }in
a Hilbert space $\mathcal{H}$ is a triple $\{\mathcal{H},\mathcal{E},%
\mathcal{D}_{\mathcal{E}}\}$, where $\mathcal{E}=\{\mathcal{H}_{\iota
};\iota \in \Upsilon \}$ is an upward filtered family of closed subspaces
such that the union space $\mathcal{D}_{\mathcal{E}}=\bigcup\limits_{\iota
\in \Upsilon }\mathcal{H}_{\iota }$ is dense in $\mathcal{H\ }$\cite{D1}.

A quantized family $\mathcal{E}=\{\mathcal{H}_{\iota };\iota \in \Upsilon \}$
determines an upward filtered family $\{P_{\iota };\iota \in \Upsilon \}$ of
projections in $B(\mathcal{H})$, where $P_{\iota }$ is a projection onto $%
\mathcal{H}_{\iota }$.

We say that a quantized domain $\mathcal{F}=\{\mathcal{K}_{\iota };\iota \in
\Upsilon \}$ \ of $\mathcal{H}$ with its union space $\mathcal{D}_{\mathcal{F%
}}$ and $\mathcal{K}$ $=\overline{\bigcup\limits_{\iota \in \Upsilon }%
\mathcal{K}_{\iota }}$ is a \textit{quantized subdomian} of $\mathcal{E}$,
if $\mathcal{K}_{\iota }\subseteq \mathcal{H}_{\iota }$ for all $\iota \in
\Upsilon $ and denote by $\mathcal{F}\subseteq $ $\mathcal{E}$.

 Let $\mathcal{E}=\{\mathcal{H}_{\iota };\iota \in \Upsilon \}$ be a
quantized domain in a Hilbert space $\mathcal{H}\ $and $\mathcal{F}=\{%
\mathcal{K}_{\iota };\iota \in \Upsilon \}$ be a quantized domain in a
Hilbert space $\mathcal{K}\ $. Then $\mathcal{E}\oplus \mathcal{F}=\left\{
\mathcal{H}_{\iota }\oplus \mathcal{K}_{\iota };\iota \in \Upsilon \right\}
\ $is a quantized domain in the Hilbert space $\mathcal{H}\oplus \mathcal{K}$
with the union space $\mathcal{D}_{\mathcal{E}\oplus \mathcal{F}}=\bigcup\limits_{\iota
\in \Upsilon }
\mathcal{H}_{\iota }\oplus \mathcal{K}_{\iota }$. For each $n,\ $we\ will
use the notations: $\underset{n}{\underbrace{\mathcal{E}\oplus ...\oplus
\mathcal{E}}}=\mathcal{E}^{\oplus n}\ $and~$\mathcal{D}_{
\mathcal{E}^{\oplus n}}=\bigcup\limits_{\iota
\in \Upsilon }
\mathcal{H}_{\iota }^{\oplus n}.$

  If $V:\mathcal{D}_{\mathcal{E}}\rightarrow \mathcal{K}$ is a linear
operator such that $V(\mathcal{H}_{\iota })\subseteq \mathcal{K}_{\iota
}$\ for all $\iota \in \Upsilon $, we write $V(\mathcal{E})\subseteq
\mathcal{F}$.

A linear operator $T$ $:$ dom$(T)$ $\subseteq $ $\mathcal{H}$ $\rightarrow $
$\mathcal{H}$ is said to be densely defined if dom$(T)$ is a dense subspace
of $\mathcal{H}$. The adjoint of $T$ is a linear map $T^{\bigstar }:$ dom$%
(T^{\bigstar })$ $\subseteq $ $\mathcal{H}\rightarrow \mathcal{H}$, where%
\begin{equation*}
\text{dom}(T^{\bigstar })=\{\xi \in \mathcal{H};\eta \rightarrow
\left\langle T\eta ,\xi \right\rangle \ \text{is continuous for every }\eta
\in \text{dom}(T)\}
\end{equation*}%
satisfying $\left\langle T\eta ,\xi \right\rangle =\left\langle \eta
,T^{\bigstar }\xi \right\rangle $ for all $\xi \in $dom$(T^{\bigstar })$ and
$\eta \in $dom$(T).$ Let $T$ $:$ dom$(T)$ $\subseteq $ $\mathcal{H}$ $%
\rightarrow $ $\mathcal{H}$ and $S$ $:$ dom$(S)$ $\subseteq $ $\mathcal{H}$ $%
\rightarrow $ $\mathcal{H}$ be two linear operator. If dom$(T)\subseteq $dom$(S)$ and $T=\left. S\right\vert _{\text{dom}(T)}$, we denote by  $T\subseteq S$.

Let $\mathcal{E}=\{\mathcal{H}_{\iota };\iota \in \Upsilon \}$ be a
quantized domain in a Hilbert space $\mathcal{H}$ and
\begin{equation*}
C(\mathcal{D}_{\mathcal{E}})=\{T\in \mathcal{L}(\mathcal{D}_{\mathcal{E}%
});TP_{\iota }=P_{\iota }TP_{\iota }\in B(\mathcal{H})\text{ for all }\iota
\in \Upsilon \}
\end{equation*}%
where $\mathcal{L}(\mathcal{D}_{\mathcal{E}})$ is the collection of all
linear operators on $\mathcal{D}_{\mathcal{E}}$. If $T\in \mathcal{L}(%
\mathcal{D}_{\mathcal{E}})$, then $T\in C(\mathcal{D}_{\mathcal{E}})$ if and
only if $T(\mathcal{H}_{\iota })\subseteq \mathcal{H}_{\mathcal{\iota }}$
and $\left. T\right\vert _{\mathcal{H}_{\iota }}\in B(\mathcal{H}_{\iota })$
for all $\iota \in \Upsilon $, and so, $C(\mathcal{D}_{\mathcal{E}})$ is an
algebra. Let
\begin{equation*}
C^{\ast }(\mathcal{D}_{\mathcal{E}})=\{T\in C(\mathcal{D}_{\mathcal{E}%
});P_{\iota }T\subseteq TP_{\iota }\text{ for all }\iota \in \Upsilon \}.
\end{equation*}%
If $T\in C(\mathcal{D}_{\mathcal{E}})$, then $T\in C^{\ast }(\mathcal{D}_{%
\mathcal{E}})$ if and only if $T(\mathcal{H}_{\iota }^{\bot }\cap \mathcal{D}%
_{\mathcal{E}})\subseteq \mathcal{H}_{\iota }^{\bot }\cap \mathcal{D}_{%
\mathcal{E}}$ for all $\iota \in \Upsilon .\ $

If $T\in C^{\ast }(\mathcal{D}_{\mathcal{E}})$, then $\mathcal{D}_{\mathcal{E%
}}$ $\subseteq $ dom$(T^{\bigstar })$. Moreover, $T^{\bigstar }(\mathcal{H}%
_{\iota })\subseteq \mathcal{H}_{\iota }$ for all $\iota \in \Upsilon $.
Now, let $T^{\ast }=\left. T^{\bigstar }\right\vert _{\mathcal{D}_{\mathcal{E%
}}}$. It is easy to check that $T^{\ast }\in C^{\ast }(\mathcal{D}_{\mathcal{%
E}}),$ and so $C^{\ast }(\mathcal{D}_{\mathcal{E}})$ is a unital $\ast $%
-algebra.

For each $\iota \in \Upsilon ,$ the map $\left\Vert \cdot \right\Vert
_{\iota }:C^{\ast }(\mathcal{D}_{\mathcal{E}})\rightarrow \lbrack 0,\infty )$%
,
\begin{equation*}
\left\Vert T\right\Vert _{\iota }=\left\Vert \left. T\right\vert _{\mathcal{H%
}_{\iota }}\right\Vert =\sup \{\left\Vert T\left( \xi \right) \right\Vert
;\xi \in \mathcal{H}_{\iota },\left\Vert \xi \right\Vert \leq 1\}
\end{equation*}%
is a $C^{\ast }$-seminorm on $C^{\ast }(\mathcal{D}_{\mathcal{E}})$.
Moreover, $C^{\ast }(\mathcal{D}_{\mathcal{E}})$ is a locally $C^{\ast }$%
-algebra with respect to the topology determined by the family of $C^{\ast }$%
-seminorms $\{\left\Vert \cdot \right\Vert _{\iota }\}_{\iota \in \Upsilon }$
and $b(C^{\ast }(\mathcal{D}_{\mathcal{E}}))$ is isomorphic with the $%
C^{\ast }$-algebra $\{T\in B\left( \mathcal{H}\right) ;P_{\iota }T=TP_{\iota
}\ $for all $\iota \in \Upsilon \}.$

If $\mathcal{E}=\{\mathcal{H}_{\iota };\iota \in \Upsilon \}$ is a quantized
domain in the Hilbert space $\mathcal{H}$, then $\mathcal{D}_{\mathcal{E}}$
can be regarded as a strict inductive limit of the direct family of Hilbert
spaces $\mathcal{E}=\{\mathcal{H}_{\iota };\iota \in \Upsilon \},$ $\mathcal{%
D}_{\mathcal{E}}=\lim\limits_{\rightarrow }\mathcal{H}_{\iota },$ and it is
called a locally Hilbert space (see \cite{I}). If $T\in C^{\ast }(\mathcal{D}%
_{\mathcal{E}})$, then $T$ is continuous \cite[Lemma 5.2 ]{I}.

A \textit{local contractive }$\ast $\textit{-morphism }from $\mathcal{A}$ to
$C^{\ast }(\mathcal{D}_{\mathcal{E}})$ is a $\ast $-morphism $\pi :\mathcal{A%
}\rightarrow $ $C^{\ast }(\mathcal{D}_{\mathcal{E}})$ with the property that
for each $\iota \in \Upsilon $ there is $\lambda \in \Lambda $ such that $%
\left\Vert  \varphi \left( a\right) \right\Vert _{\iota }\leq p_{\lambda }\left( a\right) $  for all $a\in \mathcal{A}.$ If $\Lambda =\Upsilon$ and $
\left\Vert  \varphi \left( a\right) \right\Vert _{\iota }= p_{\lambda }\left( a\right) $ for all $a\in \mathcal{A}$, we say that the  $\ast$-morphism $\pi$ is a \textit{ local isometric}  $\ast$\textit{-morphism}. 
It is known that any $\ast $-morphism between $C^{\ast }$-algebras is a
contraction. This result is not true in the case of $\ast $-morphisms
between locally $C^{\ast }$-algebras.

\begin{remark}
Let $\pi :\mathcal{A}\rightarrow $ $C^{\ast }(\mathcal{D}_{\mathcal{E}})$ be
a linear map. Then $\pi $ is a local contractive $\ast $-morphism if and only if $\pi $ is a continuous $%
\ast $-morphism. Indeed, if $\pi $ is a local contractive $\ast $-morphism,
then, clearly $\pi $ is a continuous $\ast $-morphism. Conversely, if $\pi $
is a continuous $\ast $\textit{-}morphism then, for each $\iota \in \Upsilon
,$ there exist $\lambda \in \Lambda \ $and $M_{\iota \lambda }>0$ such that $%
\left\Vert \pi \left( a\right)\right\Vert _{\iota }\leq M_{\delta \lambda }p_{\lambda
}\left( a\right) $ for all $a\in \mathcal{A}.$ Therefore, for each $\iota
\in \Upsilon ,$ there exist $\lambda \in \Lambda $ and a $\ast $-morphism $%
\pi _{\iota \lambda }:\mathcal{A}_{\lambda }\rightarrow B(\mathcal{H}_{\iota
})$\ such that $\pi _{\iota \lambda }\left( \pi _{\lambda }^{\mathcal{A}%
}(a)\right) =\left. \pi \left( a\right) \right\vert _{\mathcal{H}_{\iota }}.$
Since $\pi _{\iota \lambda }$ is a $\ast $-morphism between$\ C^{\ast }$%
-algebras, we have $\left\Vert  \pi \left( a\right) \right\Vert _{\iota }=\left\Vert \pi
_{\iota \lambda }\left( \pi _{\lambda }^{\mathcal{A}}\left( a\right) \right)
\right\Vert _{B(\mathcal{H}_{\iota })}\leq \left\Vert \pi _{\lambda }^{%
\mathcal{A}}\left( a\right) \right\Vert _{\mathcal{A}_{\lambda }}=p_{\lambda
}\left( a\right) $ for all $a\in \mathcal{A}.$ Therefore, $\pi $ is a local
contractive $\ast $-morphism.
\end{remark}
For every locally $C^{\ast }$-algebra $\mathcal{A}$ there exist a quantized
domain $\mathcal{E}$ in a Hilbert space $\mathcal{H}$ and a local isometric $%
\ast $-morphism $\pi :\mathcal{A\rightarrow }C^{\ast }(\mathcal{D}_{%
\mathcal{E}})$ \cite[Theorem 7.2]{D1}. This result can be regarded as an
unbounded analog of the Gelfand-Naimark theorem.

\paragraph{\textbf{Local completely contractive and local completely
positive maps }}

Let $\mathcal{A}$ be a unital locally $C^{\ast }$-algebra with the topology
defined by the family of $C^{\ast }$-seminorms $\left\{ p_{\lambda }\right\}
_{\lambda \in \Lambda }$ and $\{\mathcal{H},\mathcal{E},\mathcal{D}_{%
\mathcal{E}}\}$ be a quantized domain in a Hilbert space $\mathcal{H\ }$with
$\mathcal{E=\{H}_{\iota }\mathcal{\}}_{\iota \in \Upsilon }.$ 
For each $n\in \mathbb{N},$ $%
M_{n}(\mathcal{A})$ denotes the collection of all matrices of order $n$ with
elements in $\mathcal{A}$. Note that $M_{n}(\mathcal{A})$ is a locally $%
C^{\ast }$-algebra where the associated family of $C^{\ast }$-seminorms is
denoted by $\{p_{\lambda }^{n}\}_{\lambda \in \Lambda }$ and $p_{\lambda
}^{n}\left( \left[ a_{ij}\right] _{i,j=1}^{n}\right) =\left\Vert \left[ \pi
_{\lambda }^{\mathcal{A}}\left( a_{ij}\right) \right] _{i,j=1}^{n}\right%
\Vert _{M_{n}(\mathcal{A}_{\lambda })}.$ Recall that $\left[ a_{ij}\right]
_{i,j=1}^{n}\geq _{\lambda }0$ if $\left[ \pi _{\lambda }^{\mathcal{A}%
}\left( a_{ij}\right) \right] _{i,j=1}^{n}$ is a positive element in $M_{n}(%
\mathcal{A}_{\lambda })$. For each $n\in \mathbb{N}$, the locally $C^{\ast}$-algebra $ M_{n}(C^{\ast }(\mathcal{D}_{\mathcal{E}}))$ can be identified to locally $C^{\ast}$-algebra $C^{\ast }(\mathcal{D}_{\mathcal{E}^{\oplus n}})$.

For each $n\in \mathbb{N}$, the $n$-amplification of the linear map $\varphi
:\mathcal{A}\rightarrow C^{\ast }(\mathcal{D}_{\mathcal{E}})$ is the map $%
\varphi ^{\left( n\right) }:M_{n}(\mathcal{A})$ $\rightarrow $ $C^{\ast }(%
\mathcal{D}_{\mathcal{E}^{\oplus n}})$ defined by
\begin{equation*}
\varphi ^{\left( n\right) }\left( \left[ a_{ij}\right] _{i,j=1}^{n}\right) =%
\left[ \varphi \left( a_{ij}\right) \right] _{i,j=1}^{n}
\end{equation*}%
for all $\left[ a_{ij}\right] _{i,j=1}^{n}\in M_{n}(\mathcal{A})$. 

A linear map $\varphi :\mathcal{A}\rightarrow C^{\ast }(\mathcal{D}_{\mathcal{E}})$ is called :

\begin{enumerate}
\item \textit{local contractive} if for each $\iota \in \Upsilon $, there
exists $\lambda \in \Lambda $ such that%
\begin{equation*}
\left\Vert \varphi \left( a\right) \right\Vert _{\iota }\leq p_{\lambda }\left( a\right)
\text{ for all }a\in \mathcal{A};
\end{equation*}

\item \textit{positive }if\textit{\ }$\varphi \left( a\right) \geq 0$
whenever $a\geq 0;$

\item \textit{local positive} if for each $\iota \in \Upsilon $, there
exists $\lambda \in \Lambda $\ such that $\left. \varphi \left( a\right)
\right\vert _{\mathcal{H}_{\iota }}$ is positive in $B(\mathcal{H}_{\iota })$
whenever $a\geq _{\lambda }0$ and $\left. \varphi \left( a\right)
\right\vert _{\mathcal{H}_{\iota }}=0$\ \ \ whenever $a=_{\lambda }0;$

\item \textit{local completely contractive }(\textit{local }$\mathcal{CC}$%
\textit{) } \textit{\ }if for each $\iota \in \Upsilon $, there exists $%
\lambda \in \Lambda $ such that
\begin{equation*}
\left\Vert  \varphi ^{\left( n\right) }\left( \left[ a_{ij}\right]
_{i,j=1}^{n}\right) \right\Vert _{\iota }\leq
p_{\lambda }^{n}\left( \left[ a_{ij}\right] _{i,j=1}^{n}\right) \text{ }
\end{equation*}%
for all $\left[ a_{ij}\right] _{i,j=1}^{n}\in M_{n}(\mathcal{A})\ $and for
all $n\in \mathbb{N};$

\item \textit{completely positive }if $\varphi ^{\left( n\right) }\left( %
\left[ a_{ij}\right] _{i,j=1}^{n}\right) \geq 0\ $whenever $\left[ a_{ij}%
\right] _{i,j=1}^{n}\geq 0$ for all $n\in \mathbb{N};$

\item \textit{local completely positive }(\textit{local }$\mathcal{CP}$)%
\textit{\ }if for each $\iota \in \Upsilon $, there exists $\lambda \in
\Lambda $ such that $\left. \varphi ^{\left( n\right) }\left( \left[ a_{ij}%
\right] _{i,j=1}^{n}\right) \right\vert _{\mathcal{H}_{\iota }^{\oplus n}}$
is positive in $B\left( \mathcal{H}_{\iota }^{\oplus n}\right) \ $whenever $%
\left[ a_{ij}\right] _{i,j=1}^{n}\geq _{\lambda }0$ and $\left. \varphi
^{\left( n\right) }\left( \left[ a_{ij}\right] _{i,j=1}^{n}\right)
\right\vert _{\mathcal{H}_{\iota }^{\oplus n}}=0\ \ $whenever $\left[ a_{ij}%
\right] _{i,j=1}^{n}=_{\lambda }0,\ $for all $n\in \mathbb{N}.$
\end{enumerate}

If $\pi :\mathcal{A}\rightarrow C^{\ast }(\mathcal{D}_{\mathcal{E}})$ is
a local contractive $\ast $-morphism, then $\pi ^{\left( n\right)
}:M_{n}\left( \mathcal{A}\right) \rightarrow $ $M_{n}\left( C^{\ast }(%
\mathcal{D}_{\mathcal{E}})\right) $ is a local contractive $\ast $-morphism,
and so $\pi $ is a local completely positive and completely contractive.

\textbf{\ }It is known that the positivity property of the linear maps
between $C^{\ast }$-algebras implies their continuity. This property is not
true in the case of positive maps between locally $C^{\ast }$-algebras, but
it is true in the case of local positive maps \cite[Proposition 3.1]{MJ3}.
Moreover, a map between two locally $C^{\ast }$-algebras is local positive
if and only if it is positive and continuous \cite[Corollary 3.4]{MJ3}.

Let $\mathcal{A}\ $be a unital locally $C^{\ast }$-algebra and $\{\mathcal{H}%
,\mathcal{E},\mathcal{D}_{\mathcal{E}}\}$ be a quantized domain in a Hilbert
space $\mathcal{H\ }$with $\mathcal{E=\{H}_{\iota }\mathcal{\}}_{\iota \in
\Upsilon }$. The set of all local\textit{\ }$\mathcal{CPCC}$\textit{-}maps
from\textit{\ }$\mathcal{A}$ to $C^{\ast }(\mathcal{D}_{\mathcal{E}})$%
\textit{\ }is denoted by $\mathcal{CPCC}_{\text{loc}}(\mathcal{A},C^{\ast }(%
\mathcal{D}_{\mathcal{E}})).$

\begin{theorem}
\label{s} \cite[Theorem 5.1]{D1} Let $\varphi \in \mathcal{CPCC}_{\text{loc}%
}(\mathcal{A},C^{\ast }(\mathcal{D}_{\mathcal{E}}))$. Then there exist a
quantized domain $\{\mathcal{H}^{\varphi },\mathcal{E}^{\varphi },\mathcal{D}%
_{\mathcal{E}^{\varphi }}\}$, where $\mathcal{E}^{\varphi }=\{\mathcal{H}%
_{\iota }^{\varphi };\iota \in \Upsilon \}$ is an upward filtered family of
closed subspaces of $\mathcal{H}^{\varphi }$, a contraction $V_{\varphi }:%
\mathcal{H}\rightarrow \mathcal{H}^{\varphi }$ and a unital local
contractive $\ast $-morphism $\pi _{\varphi }:\mathcal{A\rightarrow }C^{\ast
}(\mathcal{D}_{\mathcal{E}^{\varphi }})$ such that

\begin{enumerate}
\item $V_{\varphi }\left( \mathcal{E}\right) \subseteq \mathcal{E}^{\varphi
};$

\item $\varphi \left( a\right) \subseteq V_{\varphi }^{\ast }\pi _{\varphi
}\left( a\right) V_{\varphi };$

\item $\mathcal{H}_{\iota }^{\varphi }=\left[ \pi _{\varphi }\left( \mathcal{%
A}\right) V_{\varphi }\mathcal{H}_{\iota }\right] $ for all $\iota \in
\Upsilon \ $\cite[Proposition 3.3]{BGK}.

Moreover, if $\varphi \left( 1_{\mathcal{A}}\right) =$id$_{\mathcal{D}_{%
\mathcal{E}}}$, then $V_{\varphi }$ is an isometry.
\end{enumerate}
\end{theorem}

The triple $\left( \pi _{\varphi },V_{\varphi },\{\mathcal{H}^{\varphi },%
\mathcal{E}^{\varphi },\mathcal{D}_{\mathcal{E}^{\varphi }}\}\right) $
constructed in Theorem \ref{s} is called a minimal Stinespring dilation of $%
\varphi $ or a minimal Stinespring representation associated to $\varphi $.
Moreover, the minimal Stinespring dilation of $\varphi $ is unique up to
unitary equivalence in the following sense: if $\left( \pi _{\varphi
},V_{\varphi },\{\mathcal{H}^{\varphi },\mathcal{E}^{\varphi },\mathcal{D}_{%
\mathcal{E}^{\varphi }}\}\right) $ and$\left( \widetilde{\pi }_{\varphi },%
\widetilde{V}_{\varphi },\{\widetilde{\mathcal{H}}^{\varphi },\widetilde{%
\mathcal{E}}^{\varphi },\widetilde{\mathcal{D}}_{\widetilde{\mathcal{E}}%
^{\varphi }}\}\right) $\ are two minimal Stinespring dilations of $\varphi $%
, then there exists a unitary operator $U_{\varphi }:\mathcal{H}^{\varphi
}\rightarrow \widetilde{\mathcal{H}}^{\varphi }$ such that $U_{\varphi
}V_{\varphi }=\widetilde{V}_{\varphi }$ and $U_{\varphi }\pi _{\varphi
}\left( a\right) \subseteq \widetilde{\pi }_{\varphi }\left( a\right)
U_{\varphi }$ for all $a\in \mathcal{A\ }$\cite[Theorem 3.4]{BGK}.

If $\varphi \in \mathcal{CPCC}_{\text{loc}}(\mathcal{A},C^{\ast }(\mathcal{D}%
_{\mathcal{E}})),$ then $\varphi \left( b(\mathcal{A})\right) \subseteq
b(C^{\ast }(\mathcal{D}_{\mathcal{E}}))$. Moreover, there is a $\mathcal{CPCC%
}$-map $\left. \varphi \right\vert _{b(\mathcal{A})}:b(\mathcal{A}%
)\rightarrow B(\mathcal{H})$ such that $\left. \left. \varphi \right\vert
_{b(\mathcal{A})}\left( a\right) \right\vert _{\mathcal{D}_{\mathcal{E}%
}}=\varphi \left( a\right) \ $for all $a\in b(\mathcal{A})\ $\cite[p. 910]%
{MJ3}, and if $\left( \pi _{\varphi },V_{\varphi },\{\mathcal{H}^{\varphi },%
\mathcal{E}^{\varphi },\mathcal{D}_{\mathcal{E}^{\varphi }}\}\right) $ is a
minimal Stinespring dilation of $\varphi $, then $\left( \left. \pi
_{\varphi }\right\vert _{b(\mathcal{A})},V_{\varphi },\mathcal{H}^{\varphi
}\right) $, where $\left. \left. \pi _{\varphi }\right\vert _{b(\mathcal{A}%
)}\left( a\right) \right\vert _{\mathcal{D}_{\widetilde{\mathcal{E}}%
^{\varphi }}}=\pi _{\varphi }\left( a\right) \ $for all $a\in b(\mathcal{A})$%
, is a minimal Stinespring dilation of $\left. \varphi \right\vert _{b(%
\mathcal{A})}$ \cite[p. 910]{MJ3}.

 \section{Extremal points}

Let $\mathcal{A}$ be a unital locally $C^{\ast }$-algebra with the topology
defined by the family of $C^{\ast }$-seminorms $\left\{ p_{\lambda }\right\}
_{\lambda \in \Lambda }$ and $\{\mathcal{H},\mathcal{E},\mathcal{D}_{%
\mathcal{E}}\}$ be a quantized domain in a Hilbert space $\mathcal{H\ }$with
$\mathcal{E=\{H}_{\iota }\mathcal{\}}_{\iota \in \Upsilon }.$

For a local contractive $\ast $-morphism $\pi :\mathcal{A\rightarrow }%
C^{\ast }(\mathcal{D}_{\mathcal{E}}),$%
\begin{equation*}
\pi \left( \mathcal{A}\right) ^{^{\prime }}=\{T\in B\left( \mathcal{H}%
\right) ;T\pi \left( a\right) \subseteq \pi \left( a\right) T\ \ \text{for\
all\ }a\in \mathcal{A}\}.
\end{equation*}

Let $\varphi \in \mathcal{CPCC}_{\text{loc}}(\mathcal{A},C^{\ast }(\mathcal{D%
}_{\mathcal{E}}))$ and $\left( \pi _{\varphi },V_{\varphi },\{\mathcal{H}%
^{\varphi },\mathcal{E}^{\varphi },\mathcal{D}_{\mathcal{E}^{\varphi
}}\}\right) $ be a minimal Stinespring dilation of $\varphi $. If $T$ is a
positive element in $\pi _{\varphi }\left( \mathcal{A}\right) ^{^{\prime
}}\cap C^{\ast }(\mathcal{D}_{\mathcal{E}^{\varphi }}),$ then the map $%
\varphi _{T}:\mathcal{A}\rightarrow C^{\ast }(\mathcal{D}_{\mathcal{E}})$
defined by $\varphi _{T}\left( a\right) =\left. V_{\varphi }^{\ast }T\pi
_{\varphi }\left( a\right) V_{\varphi }\right\vert _{\mathcal{D}_{\mathcal{E}%
}}$is local completely positive, and if $0\leq T\leq $id$_{\mathcal{H}},$
then $\varphi _{T}\in \mathcal{CPCC}_{\text{loc}}(\mathcal{A},C^{\ast }(%
\mathcal{D}_{\mathcal{E}}))$ \cite[Proposition 4.3]{BGK}.

Let $\varphi ,\psi \in \mathcal{CPCC}_{\text{loc}}(\mathcal{A},C^{\ast }(%
\mathcal{D}_{\mathcal{E}}))$. We say that $\psi $ is \textit{dominated}%
\textbf{\ }by $\varphi $, and note by $\varphi \geq \psi $, if $\varphi
-\psi \in \mathcal{CPCC}_{\text{loc}}(\mathcal{A},C^{\ast }(\mathcal{D}_{%
\mathcal{E}}))\ $ \cite[Definition 4.1]{BGK}.

If $\psi $ is dominated by $\varphi $ and, if $\left( \pi _{\varphi },V_{\varphi },\{\mathcal{H}%
^{\varphi },\mathcal{E}^{\varphi },\mathcal{D}_{\mathcal{E}^{\varphi
}}\}\right) $ is a minimal Stinespring dilation of $\varphi $, by Radon Nikodym type theorem \cite[%
Theorem 4.5]{BGK}, there is a
unique element $T\in \pi _{\varphi }\left( \mathcal{A}\right) ^{^{\prime
}}\cap C^{\ast }(\mathcal{D}_{\mathcal{E}^{\varphi }}),0\leq T\leq $id$_{%
\mathcal{H}},$ such that $\psi =\varphi _{T}.$

The following proposition is a local convex version of Corollary 1.4.4 \cite{AW}.

\begin{proposition}
\label{0} Let $\varphi \in \mathcal{CPCC}_{\text{loc}}(\mathcal{A},C^{\ast }(%
\mathcal{D}_{\mathcal{E}}))\ $and $\left( \pi _{\varphi },V_{\varphi },\{%
\mathcal{H}^{\varphi },\mathcal{E}^{\varphi },\mathcal{D}_{\mathcal{E}%
^{\varphi }}\}\right) $ be a minimal Stinespring dilation of $\varphi $.
Then the extremal points in $[0,\varphi ]=\{\psi \in \mathcal{CPCC}_{\text{%
loc}}(\mathcal{A},C^{\ast }(\mathcal{D}_{\mathcal{E}}));\varphi \geq \psi \}$
are those maps of the form $\varphi _{P}$, where $P$ is a projection in $\pi
_{\varphi }\left( \mathcal{A}\right) ^{^{\prime }}\cap C^{\ast }(\mathcal{D}%
_{\mathcal{E}^{\varphi }}).$
\end{proposition}

\begin{proof}
By \cite[Corollary 4.6]{BGK}, the map
\begin{equation*}
\zeta :\{T\in \pi _{\varphi }\left( \mathcal{A}\right) ^{^{\prime }}\cap
C^{\ast }(\mathcal{D}_{\mathcal{E}^{\varphi }}),0\leq T\leq id_{\mathcal{H}%
}\}\rightarrow \lbrack 0,\varphi ]\text{, defined by }\zeta \left( T\right)
=\varphi _{T}
\end{equation*}%
is an order isomorphism preserving convexity structure. Consequently, $%
\varphi _{T}$ is an extremal point in $[0,\varphi ]$ if and only if $T$ is
an extremal point in $\{T\in \pi _{\varphi }\left( \mathcal{A}\right)
^{^{\prime }}\cap C^{\ast }(\mathcal{D}_{\mathcal{E}^{\varphi }}),$ $0\leq
T\leq $id$_{\mathcal{H}}\}$. By \cite[Remark 4.7(1)]{MJ3}, $\{T\in \pi
_{\varphi }\left( \mathcal{A}\right) ^{^{\prime }}\cap C^{\ast }(\mathcal{D}%
_{\mathcal{E}^{\varphi }}),$ $0\leq T\leq $id$_{\mathcal{H}}\}\subseteq
\{T\in \left( \left. \pi _{\varphi }\right\vert _{b(\mathcal{A)}}\left( b(%
\mathcal{A)}\right) \right) ^{^{\prime }},$ $0\leq T\leq $id$_{\mathcal{H}%
}\} $. The extreme points in $\{T\in \left( \left. \pi _{\varphi
}\right\vert _{b(\mathcal{A)}}\left( b(\mathcal{A)}\right) \right)
^{^{\prime }},$ $0\leq T\leq $id$_{\mathcal{H}}\}$ are the projections in $%
\left( \left. \pi _{\varphi }\right\vert _{b(\mathcal{A)}}\left( b(\mathcal{%
A)}\right) \right) ^{^{\prime }}\ $\cite[Corollary 1.4.4]{AW}. If $P\in $ $%
\{T\in \pi _{\varphi }\left( \mathcal{A}\right) ^{^{\prime }}\cap C^{\ast }(%
\mathcal{D}_{\mathcal{E}^{\varphi }}),$ $0\leq T\leq $id$_{\mathcal{H}}\}$
is a projection, then $P $ is an extreme point in $\{T\in \left( \left. \pi
_{\varphi }\right\vert _{b(\mathcal{A)}}\left( b(\mathcal{A)}\right) \right)
^{^{\prime }},$ $0\leq T\leq $id$_{\mathcal{H}}\}$, and so, it is an extreme
point in $\{T\in \pi _{\varphi }\left( \mathcal{A}\right) ^{^{\prime }}\cap
C^{\ast }(\mathcal{D}_{\mathcal{E}^{\varphi }}),$ $0\leq T\leq $id$_{%
\mathcal{H}}\}.$

Conversely, let $S$ be an extreme point in $\{T\in \pi _{\varphi }\left(
\mathcal{A}\right) ^{^{\prime }}\cap C^{\ast }(\mathcal{D}_{\mathcal{E}%
^{\varphi }}),$ $0\leq T\leq $id$_{\mathcal{H}}\}$. Suppose that $S$ is not
a projection, then $S$ is not an extreme point in $\{T\in \left( \left. \pi_{\varphi }\right\vert _{b(\mathcal{A)}}\left( b(\mathcal{A)}\right) \right)
^{^{\prime }},$ $0\leq T\leq $id$_{\mathcal{H}}\}$, and so, there exist two completely positive maps $ \varphi_i:b(\mathcal{A})\rightarrow B(\mathcal{H}), i=1,2$ such that $\varphi_1 \neq \varphi_2$ and $\left( \left. \varphi
\right\vert _{b(\mathcal{A})}\right) _{S}=\frac{1}{2}\varphi _1+\frac{1%
}{2}\varphi _2$. Since $\varphi _i\leq 2\left( \left. \varphi
\right\vert _{b(\mathcal{A})}\right) _{S},i=1,2$, and $\left( \left. \varphi
\right\vert _{b(\mathcal{A})}\right) _{S}$ is continuous with respect to the
families of $C^{\ast }$-seminorms $\{\left. p_{\lambda }\right\vert _{b(%
\mathcal{A})}\}_{\lambda \in \Lambda }$ and $\{\left\Vert \cdot \right\Vert
_{\lambda }\}_{\lambda \in \Lambda },$ it follows that $\varphi _1$
and $\varphi _2$ are continuous too. Therefore, $\varphi _{i}$
extends to a local $\mathcal{CPCC}$-map from $\mathcal{A}$ to $C^{\ast }(%
\mathcal{D}_{\mathcal{E}})$, denoted by $\psi _{i},i=1,2.$ Moreover, $%
\varphi _{S}=\frac{1}{2}\psi _{1}+\frac{1}{2}\psi _{2}$. Therefore, there
exist $\psi _{1},\psi _{2}\in \mathcal{CPCC}_{\text{loc}}(\mathcal{A}%
,C^{\ast }(\mathcal{D}_{\mathcal{E}}))$, $\psi _{1}\neq \psi _{2},$ such
that $\varphi _{S}=\frac{1}{2}\psi _{1}+\frac{1}{2}\psi _{2}$, a
contradiction. Consequently, $S$ is a projection.
\end{proof}

Let $P\in C^{\ast }(\mathcal{D}_{\mathcal{E}})$ such that $0\leq P\leq $id$_{%
\mathcal{D}_{\mathcal{E}}}$. The set%
\begin{equation*}
\mathcal{CPCC}_{\text{loc,}P}(\mathcal{A},C^{\ast }(\mathcal{D}_{\mathcal{E}%
}))=\{\varphi \in \mathcal{CPCC}_{\text{loc}}(\mathcal{A},C^{\ast }(\mathcal{%
D}_{\mathcal{E}}));\varphi \left( 1_{\mathcal{A}}\right) =P\}
\end{equation*}%
is convex.

\begin{lemma}
\label{11}If $\varphi :\mathcal{A}\rightarrow C^{\ast }(\mathcal{D}_{%
\mathcal{E}})$ is a local completely positive map such that $\varphi \left(
1_{\mathcal{A}}\right) =P,$ then $\varphi \in \mathcal{CPCC}_{\text{loc,}P}(%
\mathcal{A},C^{\ast }(\mathcal{D}_{\mathcal{E}})).$
\end{lemma}

\begin{proof}
Since $\varphi $ is a local completely positive map, by \cite[Proposition
3.1 ]{MJ3}, for each $\iota \in \Upsilon ,$ there exist $\lambda \in \Lambda
$ and a completely positive map $\varphi _{\iota \lambda }:\mathcal{A}%
_{\lambda }\rightarrow B(\mathcal{H}_{\iota })$ such that $\varphi _{\iota
\lambda }\left( \pi _{\lambda }^{\mathcal{A}}\left( a\right) \right) =\left.
\varphi \left( a\right) \right\vert _{\mathcal{H}_{\iota }}.$ Moreover, by
\cite[Proposition 1.2.10]{AW},
\begin{equation*}
\left\Vert \varphi _{\iota \lambda }^{\left( n\right) }\right\Vert
=\left\Vert \varphi _{\iota \lambda }\left( \pi _{\lambda }^{\mathcal{A}%
}\left( 1_{\mathcal{A}}\right) \right) \right\Vert _{B(\mathcal{H}_{\iota
})}=\left\Vert \left. P\right\vert _{\mathcal{H}_{\iota }}\right\Vert _{B(%
\mathcal{H}_{\iota })}\leq 1
\end{equation*}%
for all $n$. Then
\begin{eqnarray*}
\left\Vert \varphi ^{\left( n\right) }\left( \left[ a_{ij}\right]
_{i,j=1}^{n}\right) \right\Vert _{\iota } &=&\left\Vert \varphi
_{\iota \lambda }^{\left( n\right) }\left( \left[ \pi _{\lambda }^{\mathcal{A%
}}\left( a_{ij}\right) \right] _{i,j=1}^{n}\right) \right\Vert _{B(\mathcal{H%
}_{\iota }^{\oplus n})} \\
&\leq &\left\Vert \varphi _{\iota \lambda }^{\left( n\right) }\right\Vert \left\Vert \left[ \pi _{\lambda }^{\mathcal{A}}\left( a_{ij}\right) %
\right] _{i,j=1}^{n}\right\Vert _{M_{n}\left( \mathcal{A}_{\lambda }\right)
}\leq p_{\lambda }^{n}\left( \left[ a_{ij}\right] _{i,j=1}^{n}\right)
\end{eqnarray*}%
for all $\left[ a_{ij}\right] _{i,j=1}^{n}\in M_{n}\left( \mathcal{A}\right)
$ and for all $n.$ Therefore, $\varphi $ is completely contractive.
\end{proof}

\begin{proposition}
\label{1}Let $\varphi \in \mathcal{CPCC}_{\text{loc,}P}(\mathcal{A},C^{\ast
}(\mathcal{D}_{\mathcal{E}}))\ $and $\left( \pi _{\varphi },V_{\varphi },\{%
\mathcal{H}^{\varphi },\mathcal{E}^{\varphi },\mathcal{D}_{\mathcal{E}%
^{\varphi }}\}\right) $ be a minimal Stinespring dilation of $\varphi $.
Then $\varphi $ is extremal in $\mathcal{CPCC}_{\text{loc,}P}(\mathcal{A}%
,C^{\ast }(\mathcal{D}_{\mathcal{E}}))$ if and only if for any $T\in \pi
_{\varphi }\left( \mathcal{A}\right) ^{^{\prime }}\cap C^{\ast }(\mathcal{D}%
_{\mathcal{E}^{\varphi }})$, the condition $V_{\varphi }^{\ast }TV_{\varphi
}=0$ implies $T=0.$
\end{proposition}

\begin{proof}
The proof is similar to the case of $C^{\ast }$-algebras \cite{AW}. Suppose
that $\varphi $ is extremal in $\mathcal{CPCC}_{\text{loc,}P}(\mathcal{A}%
,C^{\ast }(\mathcal{D}_{\mathcal{E}}))$. Let $T\in \pi _{\varphi }\left(
\mathcal{A}\right) ^{^{\prime }}\cap C^{\ast }(\mathcal{D}_{\mathcal{E}%
^{\varphi }})$ such that $V_{\varphi }^{\ast }TV_{\varphi }=0$. We can
suppose that $0\leq T\leq $id$_{\mathcal{H}}.$ Then $\varphi _{\left( \text{%
id}_{\mathcal{H}}-T\right) },$ $\varphi _{\left( \text{id}_{\mathcal{H}%
}+T\right) }\ $are local completely positive. \ Since $V_{\varphi }^{\ast
}TV_{\varphi }=0$, it follows that $\varphi _{\left( \text{id}_{\mathcal{H}%
}-T\right) }\left( 1_{\mathcal{A}}\right) =$ $\varphi _{\left( \text{id}_{%
\mathcal{H}}+T\right) }\left( 1_{\mathcal{A}}\right) =P$, and by Lemma \ref%
{11}, $\varphi _{\left( \text{id}_{\mathcal{H}}-T\right) },$ $\varphi
_{\left( \text{id}_{\mathcal{H}}+T\right) }\in \mathcal{CPCC}_{\text{loc,}P}(%
\mathcal{A},C^{\ast }(\mathcal{D}_{\mathcal{E}}))$. On the other hand, $%
\frac{1}{2}\varphi _{\left( \text{id}_{\mathcal{H}}-T\right) }+\frac{1}{2}%
\varphi _{\left( \text{id}_{\mathcal{H}}+T\right) }=\varphi $, whence, since
$\varphi $ is extremal in $\mathcal{CPCC}_{\text{loc,}P}(\mathcal{A},C^{\ast
}(\mathcal{D}_{\mathcal{E}}))$, we deduce that $\varphi _{\left( \text{id}_{%
\mathcal{H}}-T\right) }=\varphi _{\left( \text{id}_{\mathcal{H}}+T\right)
}=\varphi $. Therefore, id$_{\mathcal{H}}-T=$id$_{%
\mathcal{H}},$ \cite[Theorem 4.5]{BGK}, and then $T=0.$

Conversely,\ suppose that $\varphi $ is not extremal in $\mathcal{CPCC}_{%
\text{loc,}P}(\mathcal{A},C^{\ast }(\mathcal{D}_{\mathcal{E}}))$. Then,
there exist $\psi _{1},\psi _{2}\in \mathcal{CPCC}_{\text{loc,}P}(\mathcal{A}%
,C^{\ast }(\mathcal{D}_{\mathcal{E}})),$ $\psi _{1}\neq \psi _{2}\ $such
that $\varphi =\frac{1}{2}\psi _{1}+\frac{1}{2}\psi _{2}$.\ Since $\frac{1}{2%
}\psi _{i}\leq \varphi ,i=1,2$, by Radon Nikodym type theorem \cite[Theorem
4.5]{BGK}, there exists $T_{i}\in \pi _{\varphi }\left( \mathcal{A}\right)
^{^{\prime }}\cap C^{\ast }(\mathcal{D}_{\mathcal{E}^{\varphi }})$ such that $0\leq T_i\leq $id$_{\mathcal{H}}$ and
$\psi _{i}\left( a\right) =2\left. V_{\varphi }^{\ast }T_{i}\pi _{\varphi
}\left( a\right) V_{\varphi }\right\vert _{\mathcal{D}_{\mathcal{E}}}$ for
all $a\in \mathcal{A},i=1,2$.$\ $Then,$\ P=2V_{\varphi }^{\ast
}T_{i}V_{\varphi },i=1,2,$ whence $0=V_{\varphi }^{\ast }\left(
T_{1}-T_{2}\right) V_{\varphi },$ and consequently $T_{1}=T_{2}$. This
implies that $\psi _{1}=\psi _{2}$, a contradiction.
\end{proof}

\begin{corollary}
Let $\varphi \in \mathcal{CPCC}_{\text{loc,id}_{\mathcal{D}_{\mathcal{E}}}}(%
\mathcal{A},C^{\ast }(\mathcal{D}_{\mathcal{E}}))$. If $\varphi $ is a local
contractive $\ast $-morphism, then $\varphi $ is extremal in $\mathcal{CPCC}%
_{\text{loc,id}_{\mathcal{D}_{\mathcal{E}}}}(\mathcal{A},C^{\ast }(\mathcal{D%
}_{\mathcal{E}})).$
\end{corollary}

\section{Tensor product of unbounded local $\mathcal{CPCC}$-maps}

Let $\mathcal{A}$ and $\mathcal{B}$ be two unital locally $C^{\ast }$%
-algebras with the topologies defined by the family of $C^{\ast }$-seminorms
$\left\{ p_{\lambda }\right\} _{\lambda \in \Lambda }$ and $\left\{
q_{\delta }\right\} _{\delta \in \Delta }$, respectively, and $\{\mathcal{H},%
\mathcal{E},\mathcal{D}_{\mathcal{E}}\}$ be a quantized domain in a Hilbert
space $\mathcal{H\ }$with $\mathcal{E=\{H}_{\iota }\mathcal{\}}_{\iota \in
\Upsilon }.$

A $\ast $-representation of $\mathcal{A}$ on a Hilbert space $\mathcal{H}$
is a continuous $\ast $-morphism $\varphi :$ $\mathcal{A}$ $\rightarrow B(%
\mathcal{H})$. For each $\lambda \in \Lambda ,$ let $\mathcal{R}_{\lambda }(%
\mathcal{A})=\{\varphi \ $is a $\ast $-representation of $\mathcal{A}%
;\left\Vert \varphi \left( a\right) \right\Vert \leq p_{\lambda }(a)$ for
all $a\in \mathcal{A}\}$. The \textit{maximal tensor product} of the locally
$C^{\ast }$-algebras $\mathcal{A}$ and $\mathcal{B}$ is the locally $C^{\ast
}$-algebra $\mathcal{A}\otimes _{\max }\mathcal{B}$ obtained by the
completion of the algebraic tensor product $\mathcal{A}\otimes _{\text{alg}}%
\mathcal{B}$ with respect to the topology given by the family of $C^{\ast }$%
-seminorms $\{\upsilon _{\left( \lambda ,\delta \right) }\}_{(\lambda
,\delta )\in \Lambda \times \Delta }$, where

$\upsilon _{\left( \lambda ,\delta \right) }\left(
\sum\limits_{k=1}^{n}a_{k}\otimes b_{k}\right) =\sup \{\left\Vert \left(
\varphi \left( \sum\limits_{k=1}^{n}a_{k}\right) \psi \left(
\sum\limits_{k=1}^{n}b_{k}\right) \right) \right\Vert ;\varphi \in \mathcal{%
R}_{\lambda }(\mathcal{A}),\psi \in \mathcal{R}_{\delta }(\mathcal{B})$

\ \ \ \ \ \ \ \ \ \ \ \ \ \ \ \ \ \ \ \ \ \ \ \ \ \ \ \ \ \ \ \ \ \ \ \ \ \
\ \ \ \ \ \ \ with commuting ranges $\}.$ 

Moreover, for each $\left( \lambda ,\delta \right) \in \Lambda \times \Delta
$, the $C^{\ast }$-algebras $\left( \mathcal{A}\otimes _{\max }\mathcal{B}%
\right) _{(\lambda ,\delta )}$ and $\mathcal{A}_{\lambda }\otimes _{\max }%
\mathcal{B}_{\delta }$ are isomorphic and $\pi_{(\lambda,\delta)}^{\mathcal{A}\otimes\mathcal{B}}(a\otimes b)=\pi_{\lambda}^{\mathcal{A}}(a)\otimes\pi_{\delta}^{\mathcal{B}}(b)$. We refer the reader to \cite{F} for
further information about locally $C^{\ast }$-algebras and tensor products
of locally $C^{\ast }$-algebras.

\begin{remark}
\label{H}

\begin{enumerate}
\item Let $\varphi \in \mathcal{CPCC}_{\text{loc}}(\mathcal{A},C^{\ast }(%
\mathcal{D}_{\mathcal{E}}))$ and $\psi \in \mathcal{CPCC}_{\text{loc}}(%
\mathcal{B},C^{\ast}(\mathcal{D}_{\mathcal{E}}))$ such that $\varphi \left(
\mathcal{A}\right) $\ and $\psi \left( \mathcal{B}\right) $ commute. By \cite%
[Proposition 3.3]{MJ3}, for each $\iota \in \Upsilon $, there exist $\lambda
\in \Lambda $ and $\delta \in \Delta $ and there exist two completely
positive and completely contractive maps $\varphi _{\iota \lambda }:\mathcal{%
A}_{\lambda }\rightarrow B\left( \mathcal{H}_{\iota }\right) ,\varphi
_{\iota \lambda }\left( \pi _{\lambda }^{\mathcal{A}}\left( a\right) \right)
=\left. \varphi \left( a\right) \right\vert _{\mathcal{H}_{\iota }}$ and $%
\psi _{\iota \delta }:\mathcal{B}_{\delta }\rightarrow B\left( \mathcal{H}%
_{\iota }\right) ,\psi _{\iota \delta }\left( \pi _{\delta }^{\mathcal{B}%
}\left( b\right) \right) =\left. \psi \left( b\right) \right\vert _{\mathcal{%
H}_{\iota }}$. Moreover, $\varphi _{\iota \lambda }\left( \mathcal{A}%
_{\lambda }\right) \ $and $\psi _{\iota \delta }\left( \mathcal{B}_{\delta
}\right) \ $commute.

\item Let $\varphi :\mathcal{A}\rightarrow C^{\ast }(\mathcal{D}_{\mathcal{E}%
}))$ be a local completely positive map. As in the case of completely
positive map on $C^{\ast }$-algebras, $\varphi $ is self adjoint. Indeed,
since for each $\iota \in \Upsilon $, there exist $\lambda \in \Lambda $ and a completely positive map
$\varphi _{\iota \lambda }:\mathcal{A}_{\lambda }\rightarrow B\left(
\mathcal{H}_{\iota }\right) $ such that $\varphi _{\iota \lambda }\left( \pi
_{\lambda }^{\mathcal{A}}\left( a\right) \right) =\left. \varphi \left(
a\right) \right\vert _{\mathcal{H}_{\iota }}$ \cite[Proposition 3.3]{MJ3},
we have
\begin{equation*}
\left. \varphi \left( a^{\ast }\right) \right\vert _{\mathcal{H}_{\iota
}}=\varphi _{\iota \lambda }\left( \pi _{\lambda }^{\mathcal{A}}\left(
a^{\ast }\right) \right) =\varphi _{\iota \lambda }\left( \pi _{\lambda }^{%
\mathcal{A}}\left( a\right) ^{\ast }\right) =\varphi _{\iota \lambda }\left(
\pi _{\lambda }^{\mathcal{A}}\left( a\right) \right) ^{\ast }=\left. \varphi
\left( a\right) ^{\ast }\right\vert _{\mathcal{H}_{\iota }}.
\end{equation*}%
Consequently, $\varphi \left( a^{\ast }\right) =\varphi \left( a\right)
^{\ast }$ for all $a\in \mathcal{A}.$
\end{enumerate}
\end{remark}

\begin{lemma}
\label{max}\ Let $\varphi \in \mathcal{CPCC}_{\text{loc,id}_{\mathcal{D}_{%
\mathcal{E}}}}(\mathcal{A},C^{\ast }(\mathcal{D}_{\mathcal{E}}))$ and $\psi
\in \mathcal{CPCC}_{\text{loc,id}_{\mathcal{D}_{\mathcal{E}}}}(\mathcal{B}%
,C^{\ast }(\mathcal{D}_{\mathcal{E}}))$ such that $\varphi \left( \mathcal{A}%
\right) $\ and $\psi \left( \mathcal{B}\right) $ commute. Then there exist a
quantized domain $\{\mathcal{K},\mathcal{F},\mathcal{D}_{\mathcal{F}}\}$,\
two local contractive $\ast $-morphisms $\pi _{1}:\mathcal{A}\rightarrow $ $%
C^{\ast }(\mathcal{D}_{\mathcal{F}})$ and $\pi _{2}:\mathcal{B}\rightarrow $
$C^{\ast }(\mathcal{D}_{\mathcal{F}})$, and an isometry $V\in B\left(
\mathcal{H},\mathcal{K}\right) $ such that:

\begin{enumerate}
\item $V\left( \mathcal{E}\right) \subseteq \mathcal{F};$

\item $\pi _{1}\left( a\right) \pi _{2}\left( b\right) =\pi _{2}\left(
b\right) \pi _{1}\left( a\right) $ for all $a\in \mathcal{A}$ and for all $%
b\in \mathcal{B}$;

\item $\varphi \left( a\right) \subseteq V^{\ast }\pi _{1}\left( a\right) V\
$for all $a\in \mathcal{A};$

\item $\psi \left( b\right) \subseteq V^{\ast }\pi _{2}\left( b\right) V\ $%
for all $b\in \mathcal{B}.$
\end{enumerate}
\end{lemma}

\begin{proof}
By Remark \ref{H} $\left( 1\right) $, there exist $\lambda \in \Lambda $ and
$\delta \in \Delta $ and two completely positive maps $\varphi
_{\iota \lambda }:\mathcal{A}_{\lambda }\rightarrow B\left( \mathcal{H}%
_{\iota }\right) ,\varphi _{\iota \lambda }\left( \pi _{\lambda }^{\mathcal{A%
}}\left( a\right) \right) =\left. \varphi \left( a\right) \right\vert _{%
\mathcal{H}_{\iota }}$ and $\psi _{\iota \delta }:\mathcal{B}_{\delta
}\rightarrow B\left( \mathcal{H}_{\iota }\right) ,\psi _{\iota
\delta }\left( \pi _{\delta }^{\mathcal{B}}\left( b\right) \right) =\left.
\psi \left( b\right) \right\vert _{\mathcal{H}_{\iota }}$, such that $%
\varphi _{\iota \lambda }\left( \mathcal{A}_{\lambda }\right) \ $and $\psi
_{\iota \delta }\left( \mathcal{B}_{\delta }\right) \ $commute.

Let $\left\langle \cdot ,\cdot \right\rangle $ be the sesquilinear form on $%
\mathcal{A}\otimes _{\text{alg}}\mathcal{B}\otimes _{\text{alg}}\mathcal{D}_{%
\mathcal{E}}\ $defined by
\begin{equation*}
\left\langle \sum\limits_{j=1}^{m}c_{j}\otimes d_{j}\otimes \eta
_{j},\sum\limits_{i=1}^{n}a_{i}\otimes b_{i}\otimes \xi _{i}\right\rangle
=\sum\limits_{j=1}^{m}\sum\limits_{i=1}^{n}\left\langle \eta _{j},\varphi
\left( c_{j}^{\ast }a_{i}\right) \psi \left( d_{j}^{\ast }b_{i}\right) \xi
_{i}\right\rangle
\end{equation*}%
for all $a_{i},c_{j}\in \mathcal{A},$ for all $b_{i},d_{j}\in \mathcal{B}$
and for all $\xi _{i},\eta _{j}\in \mathcal{D}_{\mathcal{E}}$.$\ $To show
that the sesquilinear form $\left\langle \cdot ,\cdot \right\rangle $ is a
positive semidefinite form on $\mathcal{A}\otimes _{\text{alg}}\mathcal{B}%
\otimes _{\text{alg}}\mathcal{D}_{\mathcal{E}},$ let $\zeta
=\sum\limits_{i=1}^{n}a_{i}\otimes b_{i}\otimes \xi _{i}$. There is $\iota
\in \Upsilon $ such that $\xi _{i}\in \mathcal{H}_{\iota },i=1,...,n.$ Then
\begin{eqnarray*}
&&\left\langle \sum\limits_{i=1}^{n}a_{i}\otimes b_{i}\otimes \xi
_{i},\sum\limits_{i=1}^{n}a_{i}\otimes b_{i}\otimes \xi _{i}\right\rangle \\
&=&\sum\limits_{i,j=1}^{n}\left\langle \xi _{i},\varphi \left( a_{i}^{\ast
}a_{j}\right) \psi \left( b_{i}^{\ast }b_{j}\right) \xi _{j}\right\rangle \\
&=&\sum\limits_{i,j=1}^{n}\left\langle \xi _{i},\varphi _{\iota \lambda
}\left( \pi _{\lambda }^{\mathcal{A}}\left( a_{i}\right) ^{\ast }\pi
_{\lambda }^{\mathcal{A}}\left( a_{j}\right) \right) \psi _{\iota \delta
}\left( \pi _{\delta }^{\mathcal{B}}\left( b_{j}\right) ^{\ast }\pi _{\delta
}^{\mathcal{B}}\left( b_{i}\right) \right) \xi _{j}\right\rangle .
\end{eqnarray*}%
Since $\varphi _{\iota \lambda }:\mathcal{A}_{\lambda }\rightarrow B\left(
\mathcal{H}_{\iota }\right) $ and $\psi _{\iota \delta }:\mathcal{B}_{\delta
}\rightarrow B\left( \mathcal{H}_{\iota }\right) $ are completely
positive and $\varphi _{\iota \lambda }\left( \mathcal{A}_{\lambda }\right)
\ $and $\psi _{\iota \delta }\left( \mathcal{B}_{\delta }\right) \ $commute,
for each $i$ and $j$, there exist $x_{ki,}x_{kj}\in \varphi _{\iota \lambda }\left( \mathcal{A}_{\lambda }\right) ^{\prime
}=\{x\in B\left( \mathcal{H}_{\iota }\right) ;x\varphi _{\iota \lambda
}\left( a_{\lambda }\right) =\varphi _{\iota \lambda }\left( a_{\lambda
}\right) x$ for all $a_{\lambda }\in \mathcal{A}_{\lambda }$ $\},$ $k\in
\{1,...,n\}$ such that $\psi _{\iota \delta }\left( \pi _{\delta }^{\mathcal{%
B}}\left( b_{j}\right) ^{\ast }\pi _{\delta }^{\mathcal{B}}\left(
b_{i}\right) \right) =\sum\limits_{i=1}^{n}x_{ki}^{\ast }x_{kj}$ \cite[%
Proposition 4.23]{T} and then, for each $k\in \{1,...,n\},$ by \cite[Lemma
3.2]{T} , $\sum\limits_{i,j=1}^{n}\left\langle x_{ki}\xi _{i},\varphi
_{\iota \lambda }\left( \pi _{\lambda }^{\mathcal{A}}\left( a_{i}\right)
^{\ast }\pi _{\lambda }^{\mathcal{A}}\left( a_{j}\right) \right) x_{kj}\xi
_{j}\right\rangle \geq 0$ . Therefore, 
\begin{eqnarray*}
&&\left\langle \sum\limits_{i=1}^{n}a_{i}\otimes b_{i}\otimes \xi
_{i},\sum\limits_{i=1}^{n}a_{i}\otimes b_{i}\otimes \xi _{i}\right\rangle \\
&=&\sum\limits_{k=1}^{n}\left( \sum\limits_{i,j=1}^{n}\left\langle x_{ki}\xi
_{i},\varphi _{\iota \lambda }\left( \pi _{\lambda }^{\mathcal{A}}\left(
a_{i}\right) ^{\ast }\pi _{\lambda }^{\mathcal{A}}\left( a_{j}\right)
\right) x_{kj}\xi _{j}\right\rangle \right) \geq 0
\end{eqnarray*}%
and the sesquilinear form $\left\langle \cdot ,\cdot
\right\rangle $ is a positive semidefinite form on $\mathcal{A}\otimes _{%
\text{alg}}\mathcal{B}\otimes _{\text{alg}}\mathcal{D}_{\mathcal{E}}$. The quotient space $\left( \mathcal{A}\otimes _{\text{alg}}\mathcal{B}%
\otimes _{\text{alg}}\mathcal{D}_{\mathcal{E}}\right) /\mathcal{N},\ $where $%
\mathcal{N=\{\zeta \in A}\otimes _{\text{alg}}\mathcal{B}\otimes _{\text{alg}%
}\mathcal{D}_{\mathcal{E}};\left\langle \zeta ,\zeta \right\rangle =0\}$ is
a pre-Hilbert space.

Let $\mathcal{K}$ be the completion of the pre-Hilbert space $\left(
\mathcal{A}\otimes _{\text{alg}}\mathcal{B}\otimes _{\text{alg}}\mathcal{D}_{%
\mathcal{E}}\right) /\mathcal{N}$. For each $\iota \in \Upsilon ,$ let $%
K_{\iota }=$span$\{\sum\limits_{i=1}^{n}a_{i}\otimes b_{i}\otimes \xi _{i}+%
\mathcal{N};\xi _{i}\in \mathcal{H}_{\iota },a_{i}\in \mathcal{A},b_{i}\in
\mathcal{B}_{,}i=1,...,n\}$ be a subspace in $\mathcal{K}$, and let $%
\mathcal{K}_{\iota }$ be the closure of $K_{\iota }$ in $\mathcal{K}$. Then $%
\{\mathcal{K},\mathcal{F},\mathcal{D}_{\mathcal{F}}\}$ is a quantized domain
in the Hilbert space $\mathcal{K\ }$with $\mathcal{F=\{K}_{\iota }\mathcal{\}%
}_{\iota \in \Upsilon }$.

For each $a\in \mathcal{A}$, consider the linear map $\pi _{1}\left(
a\right) :$ $\mathcal{A}\otimes _{\text{alg}}\mathcal{B}\otimes _{\text{alg}}%
\mathcal{D}_{\mathcal{E}}\rightarrow $ $\mathcal{A}\otimes _{\text{alg}}%
\mathcal{B}\otimes _{\text{alg}}\mathcal{D}_{\mathcal{E}}$,$\ $defined by
\begin{equation*}
\pi _{1}\left( a\right) \left( c\otimes b\otimes \xi \right) =ac\otimes
b\otimes \xi \text{.}
\end{equation*}%
Let $\zeta =\sum\limits_{i=1}^{n}a_{i}\otimes b_{i}\otimes \xi _{i}$. Then there exists 
 $\iota \in \Upsilon $ such that $\xi _{i}\in \mathcal{H}_{\iota
},i=1,...,n\ $ and 
\begin{eqnarray*}
&&\left\langle \pi _{1}\left( a\right) \zeta ,\pi _{1}\left( a\right) \zeta
\right\rangle \\
&=&\sum\limits_{i,j=1}^{n}\left\langle \xi _{i},\varphi \left( a_{i}^{\ast
}a^{\ast }aa_{j}\right) \psi \left( b_{i}^{\ast }b_{j}\right) \xi
_{j}\right\rangle \\
&=&\sum\limits_{i,j=1}^{n}\left\langle \xi _{i},\varphi _{\iota \lambda
}\left( \pi _{\lambda }^{\mathcal{A}}\left( a_{i}\right) ^{\ast }\pi
_{\lambda }^{\mathcal{A}}\left( a\right) ^{\ast }\pi _{\lambda }^{\mathcal{A}%
}\left( a\right) \pi _{\lambda }^{\mathcal{A}}\left( a_{j}\right) \right)
\psi _{\iota \delta }\left( \pi _{\delta }^{\mathcal{B}}\left( b_{i}\right)
^{\ast }\pi _{\delta }^{\mathcal{B}}\left( b_{j}\right) \right) \xi
_{j}\right\rangle \\
&=&\sum\limits_{k=1}^{n}\left( \sum\limits_{i,j=1}^{n}\left\langle
x_{ki}\xi _{i},\varphi _{\iota \lambda }\left( \pi _{\lambda }^{\mathcal{A}%
}\left( a_{i}\right) ^{\ast }\pi _{\lambda }^{\mathcal{A}}\left( a\right)
^{\ast }\pi _{\lambda }^{\mathcal{A}}\left( a\right) \pi _{\lambda }^{%
\mathcal{A}}\left( a_{j}\right) \right) x_{kj}\xi _{j}\right\rangle \right)
\\
&=&\sum\limits_{k=1}^{n}\left\langle \left( x_{ki}\xi _{i}\right)
_{i=1}^{n},\varphi _{\iota \lambda }^{n}\left( B^{\ast }A^{\ast }AB\right)
\left( x_{ki}\xi _{j}\right) _{i=1}^{n}\right\rangle \\
\ \text{( where }B &=&\left[
\begin{array}{cccc}
\pi _{\lambda }^{\mathcal{A}}\left( a_{1}\right) & \cdot & \cdot \cdot & \pi
_{\lambda }^{\mathcal{A}}\left( a_{n}\right) \\
0 & \cdot & \cdot \cdot & 0 \\
\cdot & \cdot & \cdot \cdot & \cdot \\
0 & \cdot & \cdot \cdot & 0%
\end{array}%
\right] \text{ and }A=\left[
\begin{array}{cccc}
\pi _{\lambda }^{\mathcal{A}}\left( a\right) & \cdot & \cdot \cdot & 0 \\
0 & \cdot & \cdot \cdot & 0 \\
\cdot & \cdot & \cdot \cdot & \cdot \\
0 & \cdot & \cdot \cdot & 0%
\end{array}%
\right] \text{)} \\
&\leq &\sum\limits_{k=1}^{n}\left\Vert A\right\Vert _{M_{n}(\mathcal{A}%
_{\lambda })}^{2}\left\langle \left( x_{ki}\xi _{i}\right)
_{i=1}^{n},\varphi _{\iota \lambda }^{n}\left( B^{\ast }B\right) \left(
x_{ki}\xi _{j}\right) _{i=1}^{n}\right\rangle \\
&\leq &\left\Vert \pi _{\lambda }^{\mathcal{A}}\left( a\right) \right\Vert
^{2}\sum\limits_{k=1}^{n}\left( \sum\limits_{i,j=1}^{n}\left\langle
x_{ki}\xi _{i},\varphi _{\iota \lambda }\left( \pi _{\lambda }^{\mathcal{A}%
}\left( a_{i}\right) ^{\ast }\pi _{\lambda }^{\mathcal{A}}\left(
a_{j}\right) \right) x_{kj}\xi _{j}\right\rangle \right) \\
&\leq &\left\Vert \pi _{\lambda }^{\mathcal{A}}\left( a\right) \right\Vert
^{2}\sum\limits_{i,j=1}^{n}\left\langle \xi _{i},\varphi _{\iota \lambda
}\left( \pi _{\lambda }^{\mathcal{A}}\left( a_{i}\right) ^{\ast }\pi
_{\lambda }^{\mathcal{A}}\left( a_{j}\right) \right) \psi _{\iota \delta
}\left( \pi _{\delta }^{\mathcal{B}}\left( b_{i}\right) ^{\ast }\pi _{\delta
}^{\mathcal{B}}\left( b_{j}\right) \right) \xi _{j}\right\rangle \\
&=&p_{\lambda }\left( a\right) ^{2}\sum\limits_{i,j=1}^{n}\left\langle \xi
_{i},\varphi \left( a_{i}^{\ast }a_{j}\right) \psi \left( b_{i}^{\ast
}b_{j}\right) \xi _{j}\right\rangle \\
&=&p_{\lambda }\left( a\right) ^{2}\left\langle \zeta ,\zeta \right\rangle .
\end{eqnarray*}%
From the above relation, we conclude that $\pi _{1}\left( a\right) \left(
\mathcal{N}\right) \subseteq \mathcal{N}$. Therefore $\pi _{1}\left(
a\right) $ determines a linear operator on\ \ $\left( \mathcal{A}\otimes _{%
\text{alg}}\mathcal{B}\otimes _{\text{alg}}\mathcal{D}_{\mathcal{E}}\right) /%
\mathcal{N}$, denoted by $\pi _{1}\left( a\right) $ too. So, $\pi _{1}\left(
a\right) $ is a densely defined linear operator on $\mathcal{K}$. Moreover, $%
\left. \pi _{1}\left( a\right) \right\vert _{\mathcal{K}_{\iota }}\in B(%
\mathcal{K}_{\iota })$ and $\left\Vert \left. \pi _{1}\left( a\right)
\right\vert _{\mathcal{K}_{\iota }}\right\Vert \leq p_{\lambda _{\iota
}}\left( a\right) $. From
\begin{eqnarray*}
&&\left\langle \pi _{1}\left( a\right) \left(
\sum\limits_{i=1}^{n}a_{i}\otimes b_{i}\otimes \xi _{i}+\mathcal{N}\right)
,\left( \sum\limits_{j=1}^{m}c_{j}\otimes d_{j}\otimes \eta _{j}+\mathcal{N}%
\right) \right\rangle \\
&=&\sum\limits_{i=1}^{n}\sum\limits_{j=1}^{m}\left\langle \xi _{i},\varphi
\left( a_{i}^{\ast }a^{\ast }c_{j}\right) \psi \left( b_{i}^{\ast
}d_{j}\right) \eta _{j}\right\rangle \\
&=&\left\langle \left( \sum\limits_{i=1}^{n}a_{i}\otimes b_{i}\otimes \xi
_{i}+\mathcal{N}\right) ,\left( \sum\limits_{j=1}^{m}c_{j}\otimes a^{\ast
}d_{j}\otimes \eta _{j}+\mathcal{N}\right) \right\rangle
\end{eqnarray*}%
\begin{eqnarray*}
&=&\left\langle \left( \sum\limits_{i=1}^{n}a_{i}\otimes b_{i}\otimes \xi
_{i}+\mathcal{N}\right) ,\pi _{1}\left( a^{\ast }\right) \left(
\sum\limits_{j=1}^{m}c_{j}\otimes d_{j}\otimes \eta _{j}+\mathcal{N}\right)
\right\rangle
\end{eqnarray*}%
for all $a_{i},c_{j}\in \mathcal{A},b_{i},d_{j}\in \mathcal{B}$ and $\xi
_{i},\eta _{j}\in \mathcal{H}_{\iota }$, we conclude that $\left. \pi
_{1}\left( a^{\ast }\right) \right\vert _{\mathcal{K}_{\iota }}=\left. \pi
_{1}\left( a\right) \right\vert _{\mathcal{K}_{\iota }}^{\ast }$. Therefore,
$\pi _{1}\left( a\right) \in C^{\ast }\left( \mathcal{D}_{\mathcal{F}%
}\right) $ and $\pi _{1}\left( a^{\ast }\right) =\pi _{1}\left( a\right)
^{\ast }$. Thus, we obtained a unital $\ast $-morphism $\pi _{1}:\mathcal{A}%
\rightarrow C^{\ast }\left( \mathcal{D}_{\mathcal{F}}\right) $, which is a
local contractive map, since for each $\iota \in \Upsilon ,$ there is $%
\lambda \in \Lambda $ such that $\left\Vert \pi _{1}\left( a\right)
\right\Vert_{\iota} \leq p_{\lambda }\left(
a\right) $ for all $a\in \mathcal{A}.$

In the same way, we obtain a local contractive unital $\ast $-morphism $\pi
_{2}:\mathcal{B}\rightarrow C^{\ast }\left( \mathcal{D}_{\mathcal{F}}\right)
$, such that, for each $b\in \mathcal{B},$
\begin{equation*}
\pi _{2}\left( b\right) \left( c\otimes d\otimes \xi +\mathcal{N}\right)
=c\otimes bd\otimes \xi +\mathcal{N}\text{.}
\end{equation*}%
From
\begin{equation*}
\pi _{1}\left( a\right) \pi _{2}\left( b\right) \left( c\otimes d\otimes \xi
+\mathcal{N}\right) =ac\otimes bd\otimes \xi +\mathcal{N=}\pi _{2}\left(
b\right) \pi _{1}\left( a\right) \left( c\otimes d\otimes \xi +\mathcal{N}%
\right)
\end{equation*}%
for all $a,c\in \mathcal{A},b,d\in \mathcal{B}$ and $\xi \in \mathcal{D}_{%
\mathcal{E}},$ we deduce that $\pi _{1}\left( a\right) \pi _{2}\left(
b\right) =\pi _{2}\left( b\right) \pi _{1}\left( a\right) $ for all $a,c\in
\mathcal{A}$ and $b,d\in \mathcal{B}.$

Consider a linear map $V:\mathcal{D}_{\mathcal{E}}\rightarrow \mathcal{D}_{%
\mathcal{F}}$ defined by $V\xi =1_{\mathcal{A}}\otimes 1_{\mathcal{B}%
}\otimes \xi +\mathcal{N}.$ Clearly, $V\left( \mathcal{H}_{\iota }\right)
\subseteq \mathcal{K}_{\iota }$ for all $\iota \in \Upsilon $. For each $%
\iota \in \Upsilon $, we have
\begin{equation*}
\left\langle V\xi ,V\xi \right\rangle =\left\langle \xi ,\varphi \left( 1_{%
\mathcal{A}}\right) \psi \left( 1_{\mathcal{B}}\right) \xi \right\rangle
=\left\langle \xi ,\xi \right\rangle
\end{equation*}%
for all $\xi \in \mathcal{H}_{\iota }$. Therefore, $V$ extends to an
isometry from $\mathcal{H}$ to $\mathcal{K}$, denoted by $V$ too.

Let $a\in \mathcal{A},$ and $b\in \mathcal{B}$. From
\begin{eqnarray*}
\left\langle V^{\ast }\pi _{1}\left( a\right) V\xi ,\eta \right\rangle
&=&\left\langle \pi _{1}\left( a\right) V\xi ,V\eta \right\rangle
=\left\langle a\otimes 1_{\mathcal{B}}\otimes \xi ,1_{\mathcal{A}}\otimes 1_{%
\mathcal{B}}\otimes \eta \right\rangle \\
&=&\left\langle \xi ,\varphi \left( a^{\ast }\right) \psi \left( 1_{\mathcal{%
B}}\right) \eta \right\rangle =\left\langle \varphi \left( a\right) \xi
,\eta \right\rangle
\end{eqnarray*}%
and
\begin{eqnarray*}
\left\langle V^{\ast }\pi _{2}\left( b\right) V\xi ,\eta \right\rangle
&=&\left\langle \pi _{2}\left( b\right) V\xi ,V\eta \right\rangle
=\left\langle 1_{\mathcal{A}}\otimes b\otimes \xi ,1_{\mathcal{A}}\otimes 1_{%
\mathcal{B}}\otimes \eta \right\rangle \\
&=&\left\langle \xi ,\varphi \left( 1_{\mathcal{A}}\right) \psi \left(
b^{\ast }\right) \eta \right\rangle =\left\langle \psi \left( b\right) \xi
,\eta \right\rangle .
\end{eqnarray*}%
for all $\xi ,\eta \in \mathcal{D}_{\mathcal{E}}$, we deduce that $\varphi
\left( a\right) \subseteq V^{\ast }\pi _{1}\left( a\right) V$ and $\psi
\left( b\right) \subseteq V^{\ast }\pi _{2}\left( b\right) V$, and the lemma
is proved.
\end{proof}

\begin{proposition}
\label{tmax}Let $\varphi \in \mathcal{CPCC}_{\text{loc,id}_{\mathcal{D}_{%
\mathcal{E}}}}(\mathcal{A},C^{\ast }(\mathcal{D}_{\mathcal{E}}))$ and $\psi
\in \mathcal{CPCC}_{\text{loc,id}_{\mathcal{D}_{\mathcal{E}}}}(\mathcal{B}%
,C^{\ast }(\mathcal{D}_{\mathcal{E}}))$ such that $\varphi \left( \mathcal{A}%
\right) $\ and $\psi \left( \mathcal{B}\right) $ commute. Then there exists
a unique unital local \ $\mathcal{CPCC}$-map, $\varphi \otimes \psi :%
\mathcal{A\otimes }_{\max }\mathcal{B}\rightarrow C^{\ast }(\mathcal{D}_{%
\mathcal{E}})$, such that
\begin{equation*}
\left( \varphi \otimes \psi \right) \left( a\otimes b\right) =\varphi \left(
a\right) \psi \left( b\right)
\end{equation*}%
for all $a\in \mathcal{A}$ and for all $b\in \mathcal{B}$.
\end{proposition}

\begin{proof}
By Lemma \ref{max}, there exist a quantized domain $\{\mathcal{K},\mathcal{F}%
,\mathcal{D}_{\mathcal{F}}\}$,\ two local unital contractive $\ast $ -morphisms $%
\pi _{1}:\mathcal{A}\rightarrow $ $C^{\ast }(\mathcal{D}_{\mathcal{F}})$ and
$\pi _{2}:\mathcal{B}\rightarrow $ $C^{\ast }(\mathcal{D}_{\mathcal{F}})$,
and an isometry $V\in B\left( \mathcal{H},\mathcal{K}\right) $ such that: $%
V\left( \mathcal{E}\right) \subseteq \mathcal{F};\varphi \left( a\right)
\subseteq V^{\ast }\pi _{1}\left( a\right) V\ $for all $a\in \mathcal{A}%
;\psi \left( b\right) \subseteq V^{\ast }\pi _{2}\left( b\right) V\ $for all
$b\in \mathcal{B};\pi _{1}\left( a\right) \pi _{2}\left( b\right) =\pi
_{2}\left( b\right) \pi _{1}\left( a\right) $ for all $a\in \mathcal{A}$ and
for all $b\in \mathcal{B}$.

Since $\pi _{1}\left( a\right) \pi _{2}\left( b\right) =\pi _{2}\left(
b\right) \pi _{1}\left( a\right) $ for all $a\in \mathcal{A}$ and for all $%
b\in \mathcal{B}$, by \cite[Proposition 3.3]{Ph}, there exists a unique
local unital contractive $\ast $-morphism $\pi :\mathcal{A\otimes }_{\max }\mathcal{%
B}\rightarrow C^{\ast }(\mathcal{D}_{\mathcal{F}})$, such that
\begin{equation*}
\pi \left( a\otimes b\right) =\pi _{1}\left( a\right) \pi _{2}\left( b\right)
\end{equation*}%
for all $a\in \mathcal{A}$ and for all $b\in \mathcal{B}$. Consider the
linear map $\rho :\mathcal{A\otimes }_{\max }\mathcal{B}\rightarrow C^{\ast
}(\mathcal{D}_{\mathcal{E}})$, defined by
\begin{equation*}
\rho \left( x\right) =\left. V^{\ast }\pi \left( x\right) V\right\vert _{%
\mathcal{D}_{\mathcal{E}}}.
\end{equation*}%
Since $\pi \in \mathcal{CPCC}_{\text{loc,id}_{\mathcal{D}_{\mathcal{F}}}}(%
\mathcal{A\otimes }_{\max }\mathcal{B},C^{\ast }(\mathcal{D}_{\mathcal{F}}))$
\ and $\ V\ $ is an isometry, we deduce that $\rho \in $ \ $\mathcal{CPCC}_{\text{loc,id%
}_{\mathcal{D}_{\mathcal{E}}}}$\ $(\mathcal{A\otimes }_{\max }\mathcal{B}%
,C^{\ast }(\mathcal{D}_{\mathcal{E}}))$. Let $\xi \in \mathcal{D}_{\mathcal{E%
}},a\in \mathcal{A}$ and $b\in \mathcal{B}.$ From
\begin{eqnarray*}
\left\langle \eta ,\rho \left( a\otimes b\right) \xi \right\rangle
&=&\left\langle \eta ,V^{\ast }\pi _{1}\left( a\right) \pi _{2}\left(
b\right) V\xi \right\rangle =\left\langle V\eta ,\pi _{1}\left( a\right) \pi
_{2}\left( b\right) V\xi \right\rangle \\
&=&\left\langle 1_{\mathcal{A}}\otimes 1_{\mathcal{B}}\otimes \eta +\mathcal{%
N},\pi _{1}\left( a\right) \pi _{2}\left( b\right) \left( 1_{\mathcal{A}%
}\otimes 1_{\mathcal{B}}\otimes \xi +\mathcal{N}\right) \right\rangle \\
&=&\left\langle 1_{\mathcal{A}}\otimes 1_{\mathcal{B}}\otimes \eta +\mathcal{%
N},a\otimes b\otimes \xi +\mathcal{N}\right\rangle =\left\langle \eta
,\varphi \left( a\right) \psi \left( b\right) \xi \right\rangle
\end{eqnarray*}%
for all $\eta \in \mathcal{D}_{\mathcal{E}}$, and taking into account that $%
\mathcal{D}_{\mathcal{E}}$ is dense in $\mathcal{H}$, we deduce that $\rho
\left( a\otimes b\right) \xi =\varphi \left( a\right) \psi \left( b\right)
\xi .$ Consequently, $\rho \left( a\otimes b\right) =\varphi \left( a\right)
\psi \left( b\right) $ for all $a\in \mathcal{A}$ and for all $b\in \mathcal{%
B}.$

If $\widetilde{\rho }\in \mathcal{CPCC}_{\text{loc,id}_{\mathcal{D}_{%
\mathcal{E}}}}(\mathcal{A\otimes }_{\max }\mathcal{B},C^{\ast }(\mathcal{D}_{%
\mathcal{E}}))$ is another map such that%
\begin{equation*}
\widetilde{\rho }\left( a\otimes b\right) =\varphi \left( a\right) \psi
\left( b\right)
\end{equation*}%
for all $a\in \mathcal{A}$ and for all $b\in \mathcal{B}$. Then $\left.
\widetilde{\rho }\right\vert _{\mathcal{A\otimes }_{\text{alg}}\mathcal{B}}=$
$\left. \rho \right\vert _{\mathcal{A\otimes }_{\text{alg}}\mathcal{B}}$ and
since $\mathcal{A\otimes }_{\text{alg}}\mathcal{B}$ is dense in $\mathcal{%
A\otimes }_{\max }\mathcal{B}$, and since $\widetilde{\rho }$ and $\rho $ are
continuos (see \cite[Proposition 3.3]{MJ3}), it follows that $\widetilde{%
\rho }=\rho .$
\end{proof}

\begin{remark}
Let $\varphi \in \mathcal{CPCC}_{\text{loc,id}_{\mathcal{D}_{\mathcal{E}}}}(%
\mathcal{A},C^{\ast }(\mathcal{D}_{\mathcal{E}}))$ and $\psi \in \mathcal{%
CPCC}_{\text{loc,id}_{\mathcal{D}_{\mathcal{E}}}}(\mathcal{B},C^{\ast }(%
\mathcal{D}_{\mathcal{E}}))$ such that $\varphi \left( \mathcal{A}\right) $\
and $\psi \left( \mathcal{B}\right) $ commute. Then the triple $\left( \pi
,V,\{\mathcal{K},\mathcal{F},\mathcal{D}_{\mathcal{F}}\}\right) $
constructed in the proof of Proposition \ref{tmax} is a minimal dilation of $%
\varphi \otimes \psi .$

Indeed, we have:

\begin{enumerate}
\item $V\in B(\mathcal{H},\mathcal{K}),\ V$ is an isometry and $V\left(
\mathcal{E}\right) \subseteq \mathcal{F};$

\item $\left( \varphi \otimes \psi \right) \left( x\right) \subseteq V^{\ast
}\pi \left( x\right) V$, for all $x\in \mathcal{A\otimes }_{\max }\mathcal{B}%
;$

\item For each $\iota \in \Upsilon ,$%
\begin{eqnarray*}
\left[ \pi \left( \mathcal{A\otimes }_{\max }\mathcal{B}\right) V\mathcal{H}%
_{\iota }\right] &=&\overline{\text{span}}\{\pi \left( a\otimes b\right) V%
\mathcal{\xi };a\in \mathcal{A},b\in \mathcal{B},\mathcal{\xi \in H}_{\iota
}\} \\
&=&\overline{\text{span}}\{\pi _{1}\left( a\right) \pi _{2}\left( b\right) V%
\mathcal{\xi };a\in \mathcal{A},b\in \mathcal{B},\mathcal{\xi \in H}_{\iota
}\} \\
&=&\overline{\text{span}}\{a\otimes b\otimes \xi +\mathcal{N};a\in \mathcal{A%
},b\in \mathcal{B},\mathcal{\xi \in H}_{\iota }\}=\mathcal{K}_{\iota }
\end{eqnarray*}%
and by Theorem \ref{s}, the triple $\left( \pi ,V,\{\mathcal{K},\mathcal{F},%
\mathcal{D}_{\mathcal{F}}\}\right) $ constructed in the proof of Proposition %
\ref{tmax} is a minimal dilation of $\varphi \otimes \psi .$
\end{enumerate}
\end{remark}

\section{Marginals of a local $\mathcal{CPCC}$ -map}

Let $\mathcal{A}$ and $\mathcal{B}$ be two unital locally $C^{\ast }$%
-algebras with the topologies defined by the families of $C^{\ast }$%
-seminorms $\left\{ p_{\lambda }\right\} _{\lambda \in \Lambda }$ and $%
\left\{ q_{\delta }\right\} _{\delta \in \Delta },$ respectively, and $\{%
\mathcal{H},\mathcal{E},\mathcal{D}_{\mathcal{E}}\}$ be a quantized domain
in a Hilbert space $\mathcal{H\ }$with $\mathcal{E=\{H}_{\iota }\mathcal{\}}%
_{\iota \in \Upsilon }.$

Let $\varphi \in \mathcal{CPCC}_{\text{loc,}P}(\mathcal{A\otimes }_{\max }%
\mathcal{B},C^{\ast }(\mathcal{D}_{\mathcal{E}}))$.$\ $Clearly, the maps $%
\varphi _{\left( 1\right) }:\mathcal{A\rightarrow }C^{\ast }(\mathcal{D}_{%
\mathcal{E}})$ given by
\begin{equation*}
\varphi _{\left( 1\right) }\left( a\right) =\varphi \left( a\otimes 1_{%
\mathcal{B}}\right)
\end{equation*}%
and $\varphi _{\left( 2\right) }:\mathcal{B\rightarrow }C^{\ast }(\mathcal{D}%
_{\mathcal{E}})$ given by%
\begin{equation*}
\varphi _{\left( 2\right) }\left( b\right) =\varphi \left( 1_{\mathcal{A}%
}\otimes b\right)
\end{equation*}%
are elements in $\mathcal{CPCC}_{\text{loc,}P}(\mathcal{A},C^{\ast }(%
\mathcal{D}_{\mathcal{E}}))\ $and $\mathcal{CPCC}_{\text{loc,}P}(\mathcal{B}%
,C^{\ast }(\mathcal{D}_{\mathcal{E}})$, respectively.

\begin{definition}
Let $\varphi \in \mathcal{CPCC}_{\text{loc,}P}(\mathcal{A\otimes }_{\max }%
\mathcal{B},C^{\ast }(\mathcal{D}_{\mathcal{E}}))$.$\ $The maps $\varphi
_{\left( 1\right) }$ and $\varphi _{\left( 2\right) }$ defined the above are
called the marginals of $\varphi .$
\end{definition}

\begin{definition}
Let $\varphi _{1}\in \mathcal{CPCC}_{\text{loc,}P}(\mathcal{A},C^{\ast }(%
\mathcal{D}_{\mathcal{E}}))\ $and $\varphi _{2}\in \mathcal{CPCC}_{\text{loc,%
}P}(\mathcal{B},C^{\ast }(\mathcal{D}_{\mathcal{E}}))$. If there is $\psi \in
\mathcal{CPCC}_{\text{loc,}P}(\mathcal{A\otimes }_{\max }\mathcal{B},C^{\ast
}(\mathcal{D}_{\mathcal{E}}))$ such that $\psi _{\left( 1\right) }=\varphi
_{1}$ and $\psi _{\left( 2\right) }=\varphi _{2}$, we say that $\varphi _{1}$
and $\varphi _{2}$ are compatible and $\psi $ is a joint map for $\varphi
_{1}$ and $\varphi _{2}$.
\end{definition}

\begin{remark}
Let $\varphi \in \mathcal{CPCC}_{\text{loc,id}_{\mathcal{D}_{\mathcal{E}}}}(%
\mathcal{A},C^{\ast }(\mathcal{D}_{\mathcal{E}}))$ and $\psi \in \mathcal{%
CPCC}_{\text{loc,id}_{\mathcal{D}_{\mathcal{E}}}}(\mathcal{B},C^{\ast }(%
\mathcal{D}_{\mathcal{E}}))$ such that $\varphi \left( a\right) \psi \left(
b\right) =\psi \left( b\right) \varphi \left( a\right) $ for all $a\in
\mathcal{A}$ and for all $b\in \mathcal{B}$. Then $\varphi $ and $\psi $ are
compatible and $\varphi \otimes \psi $ (see Proposition \ref{tmax}) is a
unique joint map for $\varphi $ and $\psi .$
\end{remark}

The structure of an unbounded local $\mathcal{CPCC}$-map in terms of a
minimal Stinespring dilation of one of its marginal maps is given in the
following theorem.

\begin{theorem}
\label{Help} Let $\varphi \in \mathcal{CPCC}_{\text{loc,}P}(\mathcal{%
A\otimes }_{\max }\mathcal{B},C^{\ast }(\mathcal{D}_{\mathcal{E}}))$ and $%
(\pi _{\varphi _{\left( 1\right) }},V_{\varphi _{\left( 1\right) }},$ $\{%
\mathcal{H}^{\varphi _{\left( 1\right) }},\mathcal{E}^{\varphi _{\left(
1\right) }},$ $\mathcal{D}_{\mathcal{E}^{\varphi _{\left( 1\right) }}}\})$
be a minimal Stinespring dilation of $\varphi _{\left( 1\right) }$. Then,
there exists a unique $E\mathcal{\in CPCC}_{\text{loc,id}_{\mathcal{D}_{%
\mathcal{E}^{\varphi _{\left( 1\right) }}}}}(\mathcal{B},C^{\ast }(\mathcal{D%
}_{\mathcal{E}^{\varphi _{\left( 1\right) }}}))$\ such that $E\left(
b\right) \pi _{\varphi _{\left( 1\right) }}\left( a\right) =\pi _{\varphi
_{\left( 1\right) }}\left( a\right) E\left( b\right) $ for all $b\in
\mathcal{B}$ and for all $a\in \mathcal{A}$, and
\begin{equation*}
\varphi \left( a\otimes b\right) \subseteq V_{\varphi _{\left( 1\right)
}}^{\ast }E\left( b\right) \pi _{\varphi _{\left( 1\right) }}\left( a\right)
V_{\varphi _{\left( 1\right) }}
\end{equation*}%
for all $a\in \mathcal{A}$ and $b\in \mathcal{B}$.
\end{theorem}

\begin{proof}
Let $b\in b(\mathcal{B)}$. Consider the map $\varphi _{b}:\mathcal{%
A\rightarrow }C^{\ast }(\mathcal{D}_{\mathcal{E}})\ $defined by%
\begin{equation*}
\varphi _{b}\left( a\right) =\varphi \left( a\otimes b\right) .
\end{equation*}%
Suppose that $0\leq b\leq 1_{\mathcal{B}}$. We will show that $\varphi
_{b}\in \mathcal{CPCC}_{\text{loc}}(\mathcal{A},C^{\ast }(\mathcal{D}_{%
\mathcal{E}}))$ and $\varphi _{\left( 1\right) }\geq \varphi _{b}$.

Since $\varphi \in \mathcal{CPCC}_{\text{loc,}P}(\mathcal{A\otimes }_{\max }%
\mathcal{B},C^{\ast }(\mathcal{D}_{\mathcal{E}}))$, for each $\iota \in
\Upsilon $, there is $\left( \lambda ,\delta \right) \in \Lambda \times
\Delta $ such that $\varphi ^{\left( n\right) }\left( \left[ x_{ij}\right]
_{i,j=1}^{n}\right) \geq _{\iota }0$ whenever $\left[ x_{ij}\right]
_{i,j=1}^{n}\geq _{\left( \lambda ,\delta \right) }0$ and $\varphi ^{\left(
n\right) }\left( \left[ x_{ij}\right] _{i,j=1}^{n}\right) =_{\iota }0$
whenever $\left[ x_{ij}\right] _{i,j=1}^{n}=_{\left( \lambda ,\delta \right)
}0$ for all $n$. Let $n$ be a positive integer and $\left[ a_{ij}\right]
_{i,j=1}^{n}\in M_{n}\left( \mathcal{A}\right) $. If $\left[ a_{ij}\right]
_{i,j=1}^{n}\geq _{\lambda }0$, then $\left[ a_{ij}\otimes b\right]
_{i,j=1}^{n}\geq _{\left( \lambda ,\delta \right) }0$, and%
\begin{equation*}
\varphi _{b}^{\left( n\right) }\left( \left[ a_{ij}\right]
_{i,j=1}^{n}\right) =\left[ \varphi _{b}\left( a_{ij}\right) \right]
_{i,j=1}^{n}=\left[ \varphi \left( a_{ij}\otimes b\right) \right]
_{i,j=1}^{n}=\varphi ^{\left( n\right) }\left( \left[ a_{ij}\otimes b\right]
_{i,j=1}^{n}\right) \geq _{\iota }0.
\end{equation*}%
If $\left[ a_{ij}\right] _{i,j=1}^{n}=_{\lambda }0$, then $\left[
a_{ij}\otimes b\right] _{i,j=1}^{n}=_{\left( \lambda ,\delta \right) }0,$ and%
\begin{equation*}
\varphi _{b}^{\left( n\right) }\left( \left[ a_{ij}\right]
_{i,j=1}^{n}\right) =\varphi ^{\left( n\right) }\left( \left[ a_{ij}\otimes b%
\right] _{i,j=1}^{n}\right) =_{\iota }0.
\end{equation*}%
Thus, we showed that $\varphi _{b}$ is local completely positive.

Let $\iota \in \Upsilon .$ Since $\varphi $ is local completely contractive,
there exists $\left( \lambda ,\delta \right) \in \Lambda \times \Delta $
such that $\left\Vert \varphi ^{\left( n\right) }\left( \left[ x_{ij}%
\right] _{i,j=1}^{n}\right) 
\right\Vert _{\iota }\leq \upsilon _{\left(
\lambda ,\delta \right) }^{ n }\left( \left[ x_{ij}\right]
_{i,j=1}^{n}\right) $ for all $\left[ x_{ij}\right] _{i,j=1}^{n}\in
M_{n}\left( \mathcal{A\otimes }_{\max }\mathcal{B}\right) $ and for all $n.$
Then
\begin{eqnarray*}
\left\Vert \varphi _{b}^{\left( n\right) }\left( \left[ a_{ij}\right]
_{i,j=1}^{n}\right) \right\Vert _{\iota } &=&\left\Vert
\varphi ^{\left( n\right) }\left( \left[ a_{ij}\otimes b\right]
_{i,j=1}^{n}\right) \right\Vert _{\iota} \\
&\leq &\upsilon _{\left( \lambda ,\delta \right) }^{ n }\left(
\left[ a_{ij}\otimes b\right] _{i,j=1}^{n}\right) \leq p_{\lambda
}^{n}\left( \left[ a_{ij}\right] _{i,j=1}^{n}\right)
\end{eqnarray*}%
for all $\left[ a_{ij}\right] _{i,j=1}^{n}\in M_{n}\left( \mathcal{A}\right)
$ and for all $n$ . Therefore, $\varphi _{b}$ is local completely
contractive. Thus, we showed that $\varphi _{b}\in \mathcal{CPCC}_{\text{loc}%
}(\mathcal{A},C^{\ast }(\mathcal{D}_{\mathcal{E}}))$.

To show that $\varphi _{\left( 1\right) }\geq \varphi _{b}$, it is sufficient
to prove that $\varphi _{\left( 1\right) }-\varphi _{b}$ is a local
completely positive map. From
\begin{eqnarray*}
&&\varphi _{\left( 1\right) }^{\left( n\right) }\left( \left[ a_{ij}\right]
_{i,j=1}^{n}\right) -\varphi _{b}^{\left( n\right) }\left( \left[ a_{ij}%
\right] _{i,j=1}^{n}\right)  \\
&=&\varphi ^{\left( n\right) }\left( \left[ a_{ij}\otimes 1_{\mathcal{B}}%
\right] _{i,j=1}^{n}\right) -\varphi ^{\left( n\right) }\left( \left[
a_{ij}\otimes b\right] _{i,j=1}^{n}\right)  \\
&=&\varphi ^{\left( n\right) }\left( \left[ a_{ij}\otimes \left( 1_{\mathcal{%
B}}-b\right) \right] _{i,j=1}^{n}\right)
\end{eqnarray*}%
and taking into account that $\left[ a_{ij}\otimes \left( 1_{\mathcal{B}%
}-b\right) \right] _{i,j=1}^{n}\geq _{\left( \lambda ,\delta \right) }0$ if $%
\left[ a_{ij}\right] _{i,j=1}^{n}\geq _{\lambda }0$, and $\left[
a_{ij}\otimes \left( 1_{\mathcal{B}}-b\right) \right] _{i,j=1}^{n}=_{\left(
\lambda ,\delta \right) }0$ if $\left[ a_{ij}\right] _{i,j=1}^{n}=_{\lambda
}0$, we conclude that $\varphi _{\left( 1\right) }-\varphi _{b}$ is local
completely positive. Consequently, $\varphi _{\left( 1\right) }\geq \varphi
_{b}$ and by \cite[Theorem 4.5]{BGK}, there exists a unique $E\left(
b\right) \in \pi _{\varphi _{\left( 1\right) }}\left( \mathcal{A}\right)
^{\prime }\cap C^{\ast }(\mathcal{D}_{\mathcal{E}^{\varphi _{\left( 1\right)
}}}),$ $0\leq $ $E\left( b\right) \leq $id$_{\mathcal{H}^{\varphi _{\left(
1\right) }}},\ $such that $\varphi _{b}=\left( \varphi _{\left( 1\right)
}\right) _{E(b)}.$ In this way, we obtained a map $E:b(\mathcal{B})\mathcal{%
\rightarrow }C^{\ast }(\mathcal{D}_{\mathcal{E}^{\varphi _{\left( 1\right)
}}})\ \cap B(\mathcal{H}^{\varphi _{\left( 1\right) }})$ such that $E\left( b\right) \pi _{\varphi _{\left( 1\right) }}\left( a\right)
=\left. \pi _{\varphi _{\left( 1\right) }}\left( a\right) E\left( b\right)
\right\vert _{\mathcal{D}_{\mathcal{E}^{\varphi _{\left( 1\right) }}}}$ and 
\begin{equation*}
\varphi \left( a\otimes b\right) =\left. V_{\varphi _{\left( 1\right)
}}^{\ast }E\left( b\right) \pi _{\varphi _{\left( 1\right) }}\left( a\right)
V_{\varphi _{\left( 1\right) }}\right\vert _{\mathcal{D}_{\mathcal{E}%
^{\varphi _{\left( 1\right) }}}}
\end{equation*}%
for all $a\in \mathcal{A}$ and $b\in b(\mathcal{B)}$. Moreover, since $\varphi
_{1_{\mathcal{B}}}=\varphi _{\left( 1\right) }$, it follows that $E\left( 1_{%
\mathcal{B}}\right) =$id$_{\mathcal{D}_{\mathcal{E}^{\varphi _{\left(
1\right) }}}}$ \cite[Corollary 4.6]{BGK}.

To show that $E:b(\mathcal{B})\mathcal{\rightarrow }C^{\ast }(\mathcal{D}_{%
\mathcal{E}^{\varphi _{\left( 1\right) }}})$ is local completely positive,
let $n$ be a positive integer$\ $and $\xi _{i}\in \mathcal{H}_{\iota
}^{\varphi _{\left( 1\right) }},i=1,...,n$. Since $\left( \pi _{\varphi
_{\left( 1\right) }},V_{\varphi _{\left( 1\right) }},\{\mathcal{H}^{\varphi
_{\left( 1\right) }},\mathcal{E}^{\varphi _{\left( 1\right) }},\mathcal{D}_{%
\mathcal{E}^{\varphi _{\left( 1\right) }}}\}\right) $ is a minimal
Stinespring dilation of $\varphi _{\left( 1\right) },$ we have $\left[ \pi
_{\varphi _{\left( 1\right) }}\left( \mathcal{A}\right) V_{\varphi _{\left(
1\right) }}\mathcal{H}_{\iota }\right] =\mathcal{H}_{\iota }^{\varphi
_{\left( 1\right) }},$ and then, we can suppose that for each $i\in
\{1,...,n\}$, $\xi _{i}=\sum\limits_{k=1}^{n_{i}}\pi _{\varphi _{\left(
1\right) }}\left( a_{ik}\right) V_{\varphi _{\left( 1\right) }}\eta _{ik},$
with $\eta _{ik}\in \mathcal{H}_{\iota },$ and $a_{ik}\in \mathcal{A},$ $%
k=1,...,n_{i}$.

Let $\left[ c_{ij}\right] _{i,j=1}^{n}\in M_{n}\left( b(\mathcal{B)}\right) $
such that $\left[ c_{ij}\right] _{i,j=1}^{n}=_{\delta }0.\ $ Then $\left[
d_{ij}\otimes c_{ij}\right] _{i,j=1}^{n}$ $=_{\left( \lambda ,\delta \right)
}0$ for all $\left[ d_{ij}\right] _{i,j=1}^{n}\in M_{n}\left( \mathcal{A}%
\right) $ and $\varphi ^{\left( n\right) }\left( \left[ d_{ij}\otimes c_{ij}%
\right] _{i,j=1}^{n}\right) $ $=_{\iota }0$. From
\begin{eqnarray*}
&&\left\langle \left( \xi _{i}\right) _{i=1}^{n},E^{\left( n\right) }\left(
\left[ c_{ij}\right] _{i,j=1}^{n}\right) \left( \xi _{i}\right)
_{i=1}^{n}\right\rangle \\
&=&\sum\limits_{i,j=1}^{n}\left\langle \xi _{i},E\left( c_{ij}\right) \xi
_{j}\right\rangle \\
&=&\sum\limits_{i,j=1}^{n}\sum\limits_{k=1}^{n_{i}}\sum%
\limits_{l=1}^{n_{j}}\left\langle \pi _{\varphi _{\left( 1\right) }}\left(
a_{ik}\right) V_{\varphi _{\left( 1\right) }}\eta _{ik},E\left(
c_{ij}\right) \pi _{\varphi _{\left( 1\right) }}\left( a_{jl}\right)
V_{\varphi _{\left( 1\right) }}\eta _{jl}\right\rangle \\
&=&\sum\limits_{i,j=1}^{n}\sum\limits_{k=1}^{n_{i}}\sum%
\limits_{l=1}^{n_{j}}\left\langle \eta _{ik},V_{\varphi _{\left( 1\right)
}}^{\ast }\pi _{\varphi _{\left( 1\right) }}\left( a_{ik}^{\ast
}a_{jl}\right) E\left( c_{ij}\right) V_{\varphi _{\left( 1\right) }}\eta
_{jl}\right\rangle \\
&=&\sum\limits_{i,j=1}^{n}\sum\limits_{k=1}^{n_{i}}\sum%
\limits_{l=1}^{n_{j}}\left\langle \eta _{ik},\varphi \left( a_{ik}^{\ast
}a_{jl}\otimes c_{ij}\right) \eta _{jl}\right\rangle \\
&=&\sum\limits_{k=1}^{n_{i}}\sum\limits_{l=1}^{n_{j}}\left\langle \left(
\eta _{ik}\right) _{i=1}^{n},\varphi ^{\left( n\right) }\left( \left[
a_{ik}^{\ast }a_{jl}\otimes c_{ij}\right] _{i,j=1}^{n}\right) \left( \eta
_{jl}\right) _{j=1}^{n}\right\rangle =0
\end{eqnarray*}%
it follows that $E^{\left( n\right) }\left( \left[ c_{ij}\right]
_{i,j=1}^{n}\right) =_{\iota }0.$

Let $\left[ b_{ij}\right] _{i,j=1}^{n},\left[ c_{ij}\right] _{i,j=1}^{n}\in
M_{n}\left( b(\mathcal{B)}\right) $ such that $\left[ c_{ij}\right]
_{i,j=1}^{n}=_{\delta }0$. Then
\begin{eqnarray*}
&&\left\langle \left( \xi _{i}\right) _{i=1}^{n},E^{\left( n\right) }\left(
\left( \left[ b_{ij}\right] _{i,j=1}^{n}\right) ^{\ast }\left[ b_{ij}\right]
_{i,j=1}^{n}+\left[ c_{ij}\right] _{i,j=1}^{n}\right) \left( \xi _{i}\right)
_{i=1}^{n}\right\rangle \\
&=&\sum\limits_{i,j=1}^{n}\sum\limits_{r=1}^{n}\left\langle \xi
_{i},E\left( b_{ri}^{\ast }b_{rj}\right) \xi _{j}\right\rangle \\
&=&\sum\limits_{i,j=1}^{n}\sum\limits_{r=1}^{n}\left\langle
\sum\limits_{k=1}^{n_{i}}\pi _{\varphi _{\left( 1\right) }}\left(
a_{ik}\right) V_{\varphi _{\left( 1\right) }}\eta _{ik},E\left( b_{ri}^{\ast
}b_{rj}\right) \sum\limits_{l=1}^{n_{j}}\pi _{\varphi _{\left( 1\right)
}}\left( a_{jl}\right) V_{\varphi _{\left( 1\right) }}\eta _{jl}\right\rangle
\\
&=&\sum\limits_{i,j=1}^{n}\sum\limits_{r=1}^{n}\sum\limits_{k=1}^{n_{i}}%
\sum\limits_{l=1}^{n_{j}}\left\langle \eta _{ik},V_{\varphi _{\left(
1\right) }}^{\ast }E\left( b_{ri}^{\ast }b_{rj}\right)\pi _{\varphi _{\left( 1\right) }}\left( a_{ik}^{\ast
}a_{jl}\right) V_{\varphi _{\left(
1\right) }}\eta _{jl}\right\rangle \\
&=&\sum\limits_{i,j=1}^{n}\sum\limits_{r=1}^{n}\sum\limits_{k=1}^{n_{i}}%
\sum\limits_{l=1}^{n_{j}}\left\langle \eta _{ik},\varphi \left(
a_{ik}^{\ast }a_{jl}\otimes b_{ri}^{\ast }b_{rj}\right) \eta
_{jl}\right\rangle \\
&=&\sum\limits_{r=1}^{n}\sum\limits_{i,j=1}^{n}\sum\limits_{k=1}^{n_{i}}%
\sum\limits_{l=1}^{n_{j}}\left\langle \eta _{ik},\varphi \left( \left(
a_{ik}\otimes b_{ri}\right) ^{\ast }\left( a_{jl}\otimes b_{rj}\right)
\right) \eta _{jl}\right\rangle \\
&=&\sum\limits_{r=1}^{n}\left(
\sum\limits_{i,j=1}^{n}\sum\limits_{k=1}^{n_{i}}\sum\limits_{l=1}^{n_{j}}%
\left\langle \eta _{ik},\varphi \left( \left( d_{ik}^{r}\right) ^{\ast
}d_{jl}^{r}\right) \eta _{jl}\right\rangle \right)
\end{eqnarray*}%
where $d_{ik}^{r}=a_{ik}\otimes b_{ri}.$ Since $\varphi $ is local
completely positive, by \cite[Proposition 3.3]{MJ3}, $\varphi $ is
completely positive and then, for each $r\in \{1,...,n\}$, 
\begin{eqnarray*}
\sum\limits_{i,j=1}^{n}\sum\limits_{k=1}^{n_{i}}\sum\limits_{l=1}^{n_{j}}
\left\langle \eta _{ik},\varphi \left( \left( d_{ik}^{r}\right)^{\ast}d_{jl}^{r}\right) \eta _{jl}\right\rangle  \geq 0.
\end{eqnarray*} Therefore, 

\begin{eqnarray*}
\left\langle
\left( \xi _{i}\right) _{i=1}^{n},E^{\left( n\right) }\left( \left( \left[
b_{ij}\right] _{i,j=1}^{n}\right) ^{\ast }\left[ b_{ij}\right] _{i,j=1}^{n}+%
\left[ c_{ij}\right] _{i,j=1}^{n}\right) \left( \xi _{i}\right)
_{i=1}^{n}\right\rangle \geq 0
\end{eqnarray*}
and so, $E$ is a local $\mathcal{CP}$-map.
On the other hand, by \cite[Proposition 3.3]{MJ3}, $E$ is continuous with
respect to the families of $C^{\ast }$-seminorms $\{\left. p_{\delta
}\right\vert _{b(\mathcal{B})}\}_{\delta \in \Delta }\ $and $\{\left\Vert
\cdot \right\Vert _{\iota }\}_{\iota \in \Upsilon }$, and so, it  extends
to a unital local $\mathcal{CP}$-map from $\mathcal{B}$ to $C^{\ast }(%
\mathcal{D}_{\mathcal{E}^{\varphi _{\left( 1\right) }}})$, denoted by $E$
too. Since $E\left( 1_{\mathcal{B}}\right) =$id$_{\mathcal{D}_{\mathcal{E}%
^{\varphi _{\left( 1\right) }}}},$ by \cite[Corollary 4.1]{D1}, $E$ is also a
local $\mathcal{CC}$-map. Moreover, $E\left( b\right) \pi _{\varphi _{\left(
1\right) }}\left( a\right) =\pi _{\varphi _{\left( 1\right) }}\left(
a\right) E\left( b\right) $ for all $b\in \mathcal{B}$ and $a\in \mathcal{A}$%
, and
\begin{equation*}
\varphi \left( a\otimes b\right) =\left. V_{\varphi _{\left( 1\right)
}}^{\ast }E\left( b\right) \pi _{\varphi _{\left( 1\right) }}\left( a\right)
V_{\varphi _{\left( 1\right) }}\right\vert _{\mathcal{D}_{\mathcal{E}%
^{\varphi _{\left( 1\right) }}}}
\end{equation*}%
for all $a\in \mathcal{A}$ and $b\in \mathcal{B}.$

To prove the uniqueness of $E$, suppose that $\widetilde{E}:\mathcal{B}%
\rightarrow C^{\ast }(\mathcal{D}_{\mathcal{E}^{\varphi _{\left( 1\right)
}}})\ $is another local $\mathcal{CPCC}$-map such that $\widetilde{E}\left( b\right) \pi _{\varphi _{\left( 1\right) }}\left(
a\right) =\pi _{\varphi _{\left( 1\right) }}\left( a\right) \widetilde{E}%
\left( b\right) $  and 

\begin{equation*}
\varphi \left( a\otimes b\right) =\left. V_{\varphi _{\left( 1\right)
}}^{\ast }\widetilde{E}\left( b\right) \pi _{\varphi _{\left( 1\right)
}}\left( a\right) V_{\varphi _{\left( 1\right) }}\right\vert _{\mathcal{D}_{%
\mathcal{E}^{\varphi _{\left( 1\right) }}}}
\end{equation*}%
for all $a\in \mathcal{A}$ and $b\in \mathcal{B}$.$\ $For
each $b\in b(\mathcal{B)},$ since $\widetilde{E}\left( b\right) \in
b(C^{\ast }(\mathcal{D}_{\mathcal{E}^{\varphi _{\left( 1\right) }}})),$ $%
\widetilde{E}\left( b\right) $ extends to an element in $B(\mathcal{H}%
^{\varphi _{\left( 1\right) }})$ denoted by $\widetilde{E}\left( b\right) $
too. Therefore, $\widetilde{E}\left( b\right) \in \pi _{\varphi _{\left(
1\right) }}\left( \mathcal{A}\right) ^{\prime }\cap C^{\ast }(\mathcal{D}_{%
\mathcal{E}^{\varphi _{\left( 1\right) }}})$ and, for each $b\in b\left(
\mathcal{B}\right) \ $we have
\begin{eqnarray*}
\left( \varphi _{\left( 1\right) }\right) _{\widetilde{E}\left( b\right)
}\left( a\right) \xi &=&V_{\varphi _{\left( 1\right) }}^{\ast }\widetilde{E}%
\left( b\right) \pi _{\varphi _{\left( 1\right) }}\left( a\right) V_{\varphi
_{\left( 1\right) }}\xi =\varphi \left( a\otimes b\right) \xi \\
&=&V_{\varphi _{\left( 1\right) }}^{\ast }E\left( b\right) \pi _{\varphi
_{\left( 1\right) }}\left( a\right) V_{\varphi _{\left( 1\right) }}\xi
=\left( \varphi _{\left( 1\right) }\right) _{E\left( b\right) }\left(
a\right) \xi
\end{eqnarray*}%
for all $\xi \in \mathcal{D}_{\mathcal{E}}$ and for all $a\in \mathcal{A}$,
whence, by \cite[Theorem 4.5]{BGK}, $\widetilde{E}\left( b\right) =E\left(
b\right) $. Consequently, $\left. \widetilde{E}\right\vert _{b(\mathcal{B)}%
}=\left. E\right\vert _{b(\mathcal{B)}}$, whence, since $b(\mathcal{B)}$ is
dense in $\mathcal{B}$ and the maps $\widetilde{E}$ and $E$ are continuous,
we deduce that $\widetilde{E}=E.$
\end{proof}

The above theorem extends a result of Haapasalo, Heinosaari and Pellonp\"{a}%
\"{a} \cite[Lemma 1]{HHP} in the context of unbounded local $\mathcal{CPCC}$%
-maps. As in \cite{HHP}, the quadruple $\left( \pi _{\varphi _{\left(
1\right) }},E,V_{\varphi _{\left( 1\right) }},\{\mathcal{H}^{\varphi
_{\left( 1\right) }},\mathcal{E}^{\varphi _{\left( 1\right) }},\mathcal{D}_{%
\mathcal{E}^{\varphi _{\left( 1\right) }}}\}\right) $ is called an $\mathcal{%
A}$-subminimal dilation of $\varphi $. In the same way, we define a $%
\mathcal{B}$-subminimal dilation of $\varphi $.

\begin{theorem}
\label{C} Let $\varphi _{1}\in \mathcal{CPCC}_{\text{loc,}P}(\mathcal{A}%
,C^{\ast }(\mathcal{D}_{\mathcal{E}}))\ $and $\varphi _{2}\in \mathcal{CPCC}%
_{\text{loc,}P}(\mathcal{B},C^{\ast }(\mathcal{D}_{\mathcal{E}})$. Suppose
that $\varphi _{1}$ and $\varphi _{2}$ are compatible.

\begin{enumerate}
\item If $\varphi _{1}$ is extremal in $\mathcal{CPCC}_{\text{loc,}P}(%
\mathcal{A},C^{\ast }(\mathcal{D}_{\mathcal{E}}))$ or $\varphi _{2}$ is
extremal in

$\mathcal{CPCC}_{\text{loc,}P}(\mathcal{B},C^{\ast }(\mathcal{D}_{\mathcal{E}%
}))$, then they have a unique joint map.

\item If $\varphi _{1}$ is extremal in $\mathcal{CPCC}_{\text{loc,}P}(%
\mathcal{A},C^{\ast }(\mathcal{D}_{\mathcal{E}}))$ and $\varphi _{2}$ is
extremal in

$\mathcal{CPCC}_{\text{loc,}P}(\mathcal{B},C^{\ast }(\mathcal{D}_{\mathcal{E}%
}))$, then their unique joint map is extremal

in $\mathcal{CPCC}_{\text{loc,}P}(\mathcal{A\otimes }_{\max }\mathcal{B}%
,C^{\ast }(\mathcal{D}_{\mathcal{E}})).$

\item If $\varphi _{1}$ or $\varphi _{2}$ is a local contractive $\ast $%
-morphism, then

$\varphi _{1}\left( a\right) \varphi _{2}\left( b\right) =\varphi _{2}\left(
b\right) \varphi _{1}\left( a\right) $ for all $a\in \mathcal{A}$ and $b\in
\mathcal{B}$ and the

unique joint map $\psi \in \mathcal{CPCC}_{\text{loc,}P}(\mathcal{A\otimes }%
_{\max }\mathcal{B},C^{\ast }(\mathcal{D}_{\mathcal{E}}))$ is of the form

$\psi \left( a\otimes b\right) =\varphi _{1}\left( a\right) \varphi
_{2}\left( b\right) $ for all $a\in \mathcal{A}$ and $b\in \mathcal{B}$.
\end{enumerate}
\end{theorem}

\begin{proof}
$\left( 1\right) $\ Since $\varphi _{1}$ and $\varphi _{2}$ are compatible,
there is $\varphi \in \mathcal{CPCC}_{\text{loc,}P}(\mathcal{A\otimes }%
_{\max }\mathcal{B},C^{\ast }(\mathcal{D}_{\mathcal{E}}))$ such that $%
\varphi _{\left( 1\right) }=\varphi _{1}$ and $\varphi _{\left( 2\right)
}=\varphi _{2}$. Suppose that $\varphi $ is not unique and $\varphi _{1}$ is
extremal in $\mathcal{CPCC}_{\text{loc,}P}(\mathcal{A},C^{\ast }(\mathcal{D}%
_{\mathcal{E}}))$. Let $\psi \in \mathcal{CPCC}_{\text{loc,}P}(\mathcal{%
A\otimes }_{\max }\mathcal{B},C^{\ast }(\mathcal{D}_{\mathcal{E}}))$ be such
that $\psi _{\left( 1\right) }=\varphi _{1},$ $\psi _{\left( 2\right)
}=\varphi _{2}$ and $\varphi \neq \psi $, and $\left( \pi _{\varphi
_{1}},V_{\varphi _{1}},\{\mathcal{H}^{\varphi _{1}},\mathcal{E}^{\varphi
_{1}},\mathcal{D}_{\mathcal{E}^{\varphi _{1}}}\}\right) $ be a minimal
Stinespring dilation of $\varphi _{1}$. By Theorem \ref{Help}, there exist $%
E,\widetilde{E}\in $ $\mathcal{CPCC}_{\text{loc,}\text{id}_{\mathcal{D}_{\mathcal{E}^{{\varphi}_1}}}}(\mathcal{B},C^{\ast }(%
\mathcal{D}_{\mathcal{E}^{\varphi _{1}}}))$ such that $E\left( b\right) \pi
_{\varphi _{1}}\left( a\right) =\pi _{\varphi _{1}}\left( a\right) E\left(
b\right) $ and $\widetilde{E}\left( b\right) \pi _{\varphi _{1}}\left(
a\right) =\pi _{\varphi _{1}}\left( a\right) \widetilde{E}\left( b\right) $
for all $a\in \mathcal{A}$ and for all $b\in \mathcal{B}$ and
\begin{equation*}
\varphi \left( a\otimes b\right) =\left. V_{\varphi _{1}}^{\ast }E\left(
b\right) \pi _{\varphi _{1}}\left( a\right) V_{\varphi _{1}}\right\vert _{%
\mathcal{D}_{\mathcal{E}}}\text{ and }\psi \left( a\otimes b\right) =\left.
V_{\varphi _{1}}^{\ast }\widetilde{E}\left( b\right) \pi _{\varphi
_{1}}\left( a\right) V_{\varphi _{1}}\right\vert _{\mathcal{D}_{\mathcal{E}%
}}.
\end{equation*}%
Then
\begin{equation*}
V_{\varphi _{1}}^{\ast }\left( E\left( b\right) -\widetilde{E}\left(
b\right) \right) V_{\varphi _{1}}\xi =\varphi \left( 1_{\mathcal{A}}\otimes
b\right) \xi -\psi \left( 1_{\mathcal{A}}\otimes b\right) \xi =\varphi
_{\left( 2\right) }\left( b\right) \xi -\psi _{\left( 2\right) }\left(
b\right) \xi =0
\end{equation*}%
for all $\xi \in \mathcal{D}_{\mathcal{E}}$ and for all $b\in b(\mathcal{B)}$%
.$\ $From this relation, and taking into account that $E\left( b\right) -%
\widetilde{E}\left( b\right) \in $ $\pi _{\varphi _{1}}\left( \mathcal{A}%
\right) ^{\prime }\cap C^{\ast }(\mathcal{D}_{\mathcal{E}^{\varphi _{1}}})$
for all $b\in b(\mathcal{B)}$, and $\varphi _{1}$ is extremal in $\mathcal{%
CPCC}_{\text{loc,}P}(\mathcal{A},C^{\ast }(\mathcal{D}_{\mathcal{E}}))$,
we deduce that $E\left( b\right) =\widetilde{E}\left( b\right) $ for
all $b\in b(\mathcal{B)}$ (Proposition \ref{1}). Since $E$ and $\widetilde{E}\ $are continuous and
$b(\mathcal{B)}$ is dense in $\mathcal{B}$, from $\left. E\right\vert _{b(%
\mathcal{B})}=\left. \widetilde{E}\right\vert _{b(\mathcal{B})}$, we deduce
that $E=$ $\widetilde{E}$. Therefore, $\varphi =\psi $, a contradiction.

$\left( 2\right) $ Since $\varphi _{1}$ and $\varphi _{2}$ are extremal, the
joint map of $\varphi _{1}\ $and $\varphi _{2}$ is unique. Let $\psi \in
\mathcal{CPCC}_{\text{loc,}P}(\mathcal{A\otimes }_{\max }\mathcal{B},C^{\ast
}(\mathcal{D}_{\mathcal{E}}))$ be the joint map of $\varphi _{1}\ $and $%
\varphi _{2}.$ Suppose that $\psi $ is not extremal. Then, there exist $\psi
_{1},\psi _{2}\in \mathcal{CPCC}_{\text{loc,}P}(\mathcal{A\otimes }_{\max }%
\mathcal{B},C^{\ast }(\mathcal{D}_{\mathcal{E}})),\psi _{1}\neq \psi _{2}$
such that $\psi =\frac{1}{2}\psi _{1}+\frac{1}{2}\psi _{2}$, and
consequently, $\varphi _{i}=\psi _{\left( i\right) }=\frac{1}{2}\left( \psi
_{1}\right) _{\left( i\right) }+\frac{1}{2}\left( \psi _{2}\right) _{\left(
i\right) },i=1,2.$

Let $\left( \pi _{\varphi _{1}},V_{\varphi _{1}},\{\mathcal{H}^{\varphi
_{1}},\mathcal{E}^{\varphi _{1}},\mathcal{D}_{\mathcal{E}^{\varphi
_{1}}}\}\right) $ be a minimal Stinespring dilation of $\varphi _{1}$. From $%
\ \varphi _{1}\geq \frac{1}{2}\left( \psi _{i}\right) _{\left( 1\right)
},i=1,2$,$\ $and Radon Nikodym theorem \cite[Theorem 4.5]{BGK}, we deduce
that there exist $T_{i}\in \pi _{\varphi _{1}}\left( \mathcal{A}\right)
^{\prime }\cap C^{\ast }(\mathcal{D}_{\mathcal{E}^{\varphi _{1}}}),0\leq
T\leq $id$_{\mathcal{H}^{\varphi_1}},i=1,2$, such that
\begin{equation*}
\left( \psi _{i}\right) _{\left( 1\right) }\left( a\right) =\left.
2V_{\varphi _{1}}^{\ast }T_{i}\pi _{\varphi _{1}}\left( a\right) V_{\varphi
_{1}}\right\vert _{\mathcal{D}_{\mathcal{E}}},i=1,2.
\end{equation*}%
Let $S_{i}=2T_{i},$ $i=1,2.$ Clearly, $S_{i}\in \pi _{\varphi _{1}}\left(
\mathcal{A}\right) ^{\prime }\cap C^{\ast }(\mathcal{D}_{\mathcal{E}%
^{\varphi _{1}}})$ $,i=1,2.$

For each $\iota \in \Upsilon $, put $\left[ \left( S_{i}\right) ^{\frac{1}{2}%
}\mathcal{H}_{\iota }^{\varphi _{1}}\right] =\mathcal{H}_{\iota }^{i},i=1,2$%
. Clearly, $\mathcal{H}_{\iota _{1}}^{i}\subseteq \mathcal{H}_{\iota
_{2}}^{i}\ $for all $\iota _{1},\iota _{2}\in \Upsilon \ $with $\iota
_{1}\leq \iota _{2}$, $i=1,2$. Let $\mathcal{H}^{i}=\left[ \left(
S_{i}\right) ^{\frac{1}{2}}\mathcal{H}^{\varphi _{1}}\right] ,$ $i=1,2.$
Then, $\{\mathcal{H}^{i};\mathcal{E}^{i};\mathcal{D}_{\mathcal{E}^{i}}\}$,
where $\mathcal{E}^{i}\mathcal{=\{H}_{\iota }^{i}\mathcal{\}}_{\iota \in
\Upsilon }$, is a quantized domain in the Hilbert space $\mathcal{H}^{i},%
\mathcal{\ }i=1,2.\ $We have :

$(a)\ \left( S_{i}\right) ^{\frac{1}{2}}V_{\varphi _{1}}\in B(\mathcal{H},%
\mathcal{H}^{i}),$ $\left( S_{i}\right) ^{\frac{1}{2}}V_{\varphi _{1}}\left(
\mathcal{E}\right) \subseteq \mathcal{E}^{i};$

$(b)$\ $\left( \psi _{i}\right) _{\left( 1\right) }\left( a\right) =\left.
V_{\varphi _{1}}^{\ast }\left( S_{i}\right) ^{\frac{1}{2}}\pi _{\varphi
_{1}}^{i}\left( a\right) \left( S_{i}\right) ^{\frac{1}{2}}V_{\varphi
_{1}}\right\vert _{\mathcal{D}_{\mathcal{E}}}$,\ \ \ for all $a\in \mathcal{A%
}$, where $\pi _{\varphi _{1}}^{i}\left( a\right) =\left. \pi _{\varphi
_{1}}\left( a\right) \right\vert _{\mathcal{D}_{\mathcal{E}^{i}}};$

$(c)\ \ \left[ \pi _{\varphi _{1}}^{i}\left( \mathcal{A}\right) \left(
S_{i}\right) ^{\frac{1}{2}}V_{\varphi _{1}}\mathcal{H}_{_{\iota }}\right] =%
\left[ \left( S_{i}\right) ^{\frac{1}{2}}\pi _{\varphi _{1}}\left( \mathcal{A%
}\right) V_{\varphi _{1}}\mathcal{H}_{\iota }\right] =$ $\left[ \left(
S_{i}\right) ^{\frac{1}{2}}\mathcal{H}_{\iota }^{\varphi _{1}}\right] =%
\mathcal{H}_{\iota }^{i}$ for all $\iota \in \Upsilon $.

From the above relations and Theorem \ref{s}, we deduce that $\left( \pi
_{\varphi _{1}}^{i},\left( S_{i}\right) ^{\frac{1}{2}}V_{\varphi _{1}},\{%
\mathcal{H}^{i},\mathcal{E}^{i},\mathcal{D}_{\mathcal{E}^{i}}\}\right) $, $%
i=1,2$, is a minimal Stinespring dilation of $\left( \psi _{i}\right)
_{\left( 1\right) },i=1,2$. Then, by Theorem \ref{Help}, there exist
$E_{i}\in $ $\mathcal{CPCC}_{\text{loc,id}_{\mathcal{D}_{\mathcal{E}^{i}}}}(%
\mathcal{B},C^{\ast }(\mathcal{D}_{\mathcal{E}^{i}}))$ such that $%
E_{i}\left( b\right) \in \pi _{\varphi _{1}}\left( \mathcal{A}\right)
^{\prime }\cap C^{\ast }(\mathcal{D}_{\mathcal{E}^{i}})$ and 
\begin{eqnarray*}
\psi _{i}\left( a\otimes b\right)  &=&\left. V_{\varphi _{1}}^{\ast }\left(
S_{i}\right) ^{\frac{1}{2}}E_{i}\left( b\right) \pi _{\varphi
_{1}}^{i}\left( a\right) \left( S_{i}\right) ^{\frac{1}{2}}V_{\varphi
_{1}}\right\vert _{\mathcal{D}_{\mathcal{E}}} \\
&=&\left. V_{\varphi _{1}}^{\ast }\left( S_{i}\right) ^{\frac{1}{2}%
}E_{i}\left( b\right) \left( S_{i}\right) ^{\frac{1}{2}}\pi _{\varphi
_{1}}\left( a\right) V_{\varphi _{1}}\right\vert _{\mathcal{D}_{\mathcal{E}}}
\end{eqnarray*}%
for all $a\in \mathcal{A}$ and for all $b\in b(\mathcal{B)},i=1,2$. 
 Since $\varphi _{2}\ $is extremal, from $\varphi
_{2}=\psi _{\left( 2\right) }=\frac{1}{2}\left( \psi _{1}\right) _{\left(
2\right) }+\frac{1}{2}\left( \psi _{2}\right) _{\left( 2\right) }$, we
deduce that $\left( \psi _{1}\right) _{\left( 2\right) }=\left( \psi
_{2}\right) _{\left( 2\right) }$. Therefore, for each $b\in b(\mathcal{B)},$
\begin{eqnarray*}
0 &=&\left( \psi _{1}\right) _{\left( 2\right) }\left( b\right) -\left( \psi
_{2}\right) _{\left( 2\right) }\left( b\right) \\
&=&\left. V_{\varphi _{1}}^{\ast }\left( \left( S_{1}\right) ^{\frac{1}{2}%
}E_{1}\left( b\right) \left( S_{1}\right) ^{\frac{1}{2}}-\left( S_{2}\right)
^{\frac{1}{2}}E_{2}\left( b\right) \left( S_{2}\right) ^{\frac{1}{2}}\right)
V_{\varphi _{1}}\right\vert _{\mathcal{D}_{\mathcal{E}}}.
\end{eqnarray*}%
From the above relation, \ \ and \ taking\ into\ \ account\ that $\ \left(
S_{1}\right) ^{\frac{1}{2}}E_{1}\left( b\right) \left( S_{1}\right) ^{\frac{1%
}{2}}-$ $\ \left( S_{2}\right) ^{\frac{1}{2}}E_{2}\left( b\right) \left(
S_{2}\right) ^{\frac{1}{2}}\in \pi _{\varphi _{1}}\left( \mathcal{A}\right)
^{\prime }\cap C^{\ast }(\mathcal{D}_{\mathcal{E}})$ $\ $for all $b\in b(%
\mathcal{B})$ and $\varphi _{1}\ $is extremal, we
deduce that $\left( S_{1}\right) ^{\frac{1}{2}}E_{1}\left( b\right) \left(
S_{1}\right) ^{\frac{1}{2}}=\left( S_{2}\right) ^{\frac{1}{2}}E_{2}\left(
b\right) \left( S_{2}\right) ^{\frac{1}{2}}$ for all $b\in b(\mathcal{B)}$ (Proposition \ref{1}).
Consequently, $\psi _{1}\left( a\otimes b\right) =\psi _{2}\left( a\otimes
b\right) $ for all $a\in \mathcal{A}$ and for all $b\in b\left( \mathcal{B}%
\right) ,$ and since $\psi _{1}$ and $\psi _{2}$ are continuous, $\psi _{1}$ $%
=$ $\psi _{2},$ a contradiction. Therefore, $\psi $ is extremal.

$\left( 3\right) \ $Suppose that $\varphi _{1}$ is a local contractive $\ast
$-morphism. Then $P\ $is a projection in $\varphi _{1}\left( \mathcal{A}%
\right) ^{^{\prime }}\cap C^{\ast }(\mathcal{D}_{\mathcal{E}})$. Since $%
\left[ \varphi _{1}\left( \mathcal{A}\right) P\mathcal{H}_{\iota }\right] =P%
\mathcal{H}_{\iota }$ and $\left[ \varphi _{1}\left( \mathcal{A}\right)
\left( \text{id}_{\mathcal{H}}-P\right) \mathcal{H}_{\iota }\right] =\{0\}$
for all $\iota \in \Upsilon ,$ we can suppose that, for each $\iota \in
\Upsilon ,P\mathcal{H}_{\iota }=\mathcal{H}_{\iota }$. Then $\left[ \varphi
_{1}\left( \mathcal{A}\right) \mathcal{H}_{\iota }\right] =\mathcal{H}%
_{\iota }$ for all $\iota \in \Upsilon $ and $\left( \varphi _{1},\text{id}_{%
\mathcal{D}_{\mathcal{E}}},\{\mathcal{H},\mathcal{E},\mathcal{D}_{\mathcal{E}%
}\}\right) $ is a minimal Stinespring dilation of $\varphi _{1}$. Since $%
\varphi _{1}\ $is extremal\ in $\mathcal{CPCC}_{\text{loc,}P}(\mathcal{A}%
,C^{\ast }(\mathcal{D}_{\mathcal{E}}))$, by $(1)$, there exists a unique
joint map $\psi \in \mathcal{CPCC}_{\text{loc,}P}(\mathcal{A\otimes }_{\max }%
\mathcal{B},C^{\ast }(\mathcal{D}_{\mathcal{E}}))$, and by Theorem \ref{Help}%
, there exists a unique $E\in \mathcal{CPCC}_{\text{loc,id}_{\mathcal{D}_{%
\mathcal{E}}}}(\mathcal{B},C^{\ast }(\mathcal{D}_{\mathcal{E}})) $\ such
that $E\left( b\right) \varphi _{1}\left( a\right) =\varphi _{1}\left(
a\right) E\left( b\right) $ for all $b\in \mathcal{B}$ and for all $a\in
\mathcal{A}$, and
\begin{equation*}
\psi \left( a\otimes b\right) =\varphi _{1}\left( a\right) E\left( b\right)
=E\left( b\right) \varphi _{1}\left( a\right)
\end{equation*}%
for all $a\in \mathcal{A}$ and for all $b\in \mathcal{B}$. From the above
relation, we deduce that $\varphi _{2}=E,$ and so $\varphi _{1}$ and $%
\varphi _{2}$ commute.
\end{proof}

\begin{corollary}
Let $\varphi _{1}:\mathcal{A}\rightarrow C^{\ast }(\mathcal{D}_{\mathcal{E}%
}))\ $be a unital local contractive $\ast $-morphism and $\varphi _{2}:%
\mathcal{B}\rightarrow C^{\ast }(\mathcal{D}_{\mathcal{E}})\ $be a unital $%
\mathcal{CPCC}$-map. Then, $\varphi _{1}$ and $\varphi _{2}$ are compatible
if and only if the ranges of $\varphi _{1}$ and $\varphi _{2}\ $commute.
\end{corollary}

\begin{acknowledgement}
The author would like to thank the referee for his/her useful comments.
\end{acknowledgement}


\begin{thebibliography}{99}
\bibitem{A} C. Apostol, $b^{\ast }$\textit{-algebras and their representation%
}, J. LondonMath.Soc. \textbf{3}(1971), 30-38.

\bibitem{AW} W.B. Arveson, \textit{Subalgebras of }$C^{\ast }$\textit{%
-algebras}, Acta Math. \textbf{123}(1969), 141-224.

\bibitem{BGK} B. V. R. Bhat, A. Ghatak and S. K. Pamula, \textit{%
Stinespring's theorem for unbounded operator valued local completely
positive maps and its applications, }Indag. Math., \textbf{32}(2021),2,
547-578.

\bibitem{D1} A. Dosiev, \textit{Local operator spaces, unbounded operators
and multinormed }$C^{\ast }$\textit{-algebras, }J. Funct. Anal. \textbf{255}%
(2008), 1724--1760.

\bibitem{F} M. Fragoulopoulou, Topological algebras with involution,
Elsevier, 2005.

\bibitem{HHP} E. Haapasalo, T. Heinosaari, J-P. Pellonp\"{a}\"{a}, \textit{%
When do pieces determine the whole? Extremal marginals of a completely
positive map, }Rev. Math. Phys. \textbf{26}(2014),2,1450002.

\bibitem{I} A. Inoue, \textit{Locally }$C^{\ast }$\textit{-algebras}, Mem.
Fac. Sci. Kyushu Univ. Ser. A, \textbf{25} (1971), 197--235.

\bibitem{MJ1} M. Joi\c{t}a, \textit{Unbounded local completely positive maps
of local order zero}, Positivity, \textbf{25}(2021),3,1215-1227.

\bibitem{MJ3} M. Joi\c{t}a, \textit{Factorization properties for unbounded
local positive maps, }Indag. Math., \textbf{33}(2022), 4,905-917.

\bibitem{MJ2} M. Joi\c{t}a, \textit{Unbounded local completely contractive
maps, }Bull. Iranian Math. Soc., \textbf{48}(2022),6, 4015-4028.

\bibitem{Ph} N.C.\ Phillips, \textit{Inverse limits of }$C^{\ast }$\textit{%
-algebras}, J. Operator Theory, \textbf{19} (1988), 1, 159--195.

\bibitem{S} K. Schm\"{u}dgen, \textit{\"{U}ber }$LMC^{\ast }$\textit{%
-Algebras}, Math.Nachr., \textbf{68}(1975),167-182.

\bibitem{T} M. Takesaki, \textit{Theory of Operator Algebra I, }Encyclopedia
of Math. Sci., 124, Springer, 2002.

\bibitem{V} D. Voiculescu, \textit{Dual algebraic structure on operator
algebras related to free products}, J. Operator Theory, 17(1987),85-98.
\end{thebibliography}
\end{document}